\documentclass[12pt,a4paper]{amsart}
\allowdisplaybreaks 

\usepackage{color}
\usepackage{amsmath}
\usepackage{amssymb}
\usepackage{amsthm}
\usepackage{amsfonts}
\usepackage{latexsym}
\usepackage{hyperref}
\usepackage{tikz}
\usetikzlibrary{cd}
\usepackage{enumerate}
\usepackage{ulem}

\usetikzlibrary{matrix,arrows}

\usepackage{mathtools}
\usepackage{gensymb}

\newtheorem{theorem}{Theorem}[section]
\newtheorem{corollary}[theorem]{Corollary}
\newtheorem{lemma}[theorem]{Lemma}
\newtheorem{proposition}[theorem]{Proposition}
\theoremstyle{definition}

\numberwithin{equation}{section}
\newcommand\pref[1]{~(\ref{#1})}

\def \bC {\mathbb C}
\def \bD {\mathbb D}

\def \bH {\mathbb H}

\def \bN {\mathbb N}

\def \bR {\mathbb R}

\def \bR {\mathbb R}
\def \bR {\mathbb R}

\def \bZ {\mathbb Z}
\def \DD {\mathbf D}

\def \cC {\mathcal C}
\def \cD {\mathcal D}

\def \cG {\mathcal G}
\def \cH {\mathcal H}

\def \cP {\mathcal P}

\def \cS {\mathcal S}

\def \cW {\mathcal W}

\def \fg {\mathfrak g}

\def \fk {\mathfrak k}

\def \fo {\mathfrak o}

\def \fs {\mathfrak s}

\def \fu {\mathfrak u}

\def \fU {\mathfrak U}

\def \RE {\text{\rm Re}\,}
\def \IM {\text{\rm Im}\,}
\def \al {\alpha}
\def \la {\lambda}
\def \ph {\varphi}

\def \eps {\varepsilon}
\def \lan {\langle}
\def \ran {\rangle}
\def \de {\partial}
\def \spaz {\text{\rm span\,}}

\def \half{\frac12}
\def \inv{^{-1}}

\def \deg {\text{\rm deg\,}}

\def \tr {\text{\rm tr\,}}

\def\be{\begin{equation}}
\def\ee{\end{equation}}
\def\bes{\begin{equation*}}
\def\ees{\end{equation*}}
\def\bea{\begin{equation}\begin{aligned}}
\def\eea{\end{aligned}\end{equation}}
\def\beas{\begin{equation*}\begin{aligned}}
\def\eeas{\end{aligned}\end{equation*}}

\newcommand\grpSU{{SU_2}}
\newcommand\grpSO{{SO}}

\newcommand\base{{o}} 

\newcommand\baseuno {X_1}

\newcommand\als[1]{\begin{align*} #1 \end{align*}}
\newcommand\diagonali{\mathcal{B}}
\newcommand\spaziorp{\mathcal{V}}
\newcommand\spaziorpU{\mathcal{U}}
\newcommand\spaziW{\mathcal{W}}
\newcommand\fine{a}
\newcommand\finetau{{\fine_\tau}}
\newcommand\enne{\nu} 
\newcommand\emme{\mu}

\newcommand\Btau{B_\tau}
\newcommand\Smatr{\cS\big(\bR^n,{\rm End}(\spaziorp_\tau)\big)^K}

\newcommand\SmatrH{\cS\big(H,{\rm End}(\spaziorp_\tau)\big)^K}
\newcommand\SmatrW{\cS(\bR^n,\spaziW_\tau^\ell)^K}

\newcommand\LmatrH{L^1\big(H,{\rm End}(\spaziorp_\tau)\big)^K}
\newcommand\LduematrH{L^2\big(H,{\rm End}(\spaziorp_\tau)\big)^K}

\newcommand\Stau{\cS(G)_{\tau}^{{\rm Int}(K)}}

\newcommand\ind{d} 
 
\newcommand\der{m} 
\newcommand\undue{\kappa}
\newcommand\Fo{F_-} 
\newcommand\Fe{F_+} 
\newcommand\ro{\rho} 

\def\aut#1{\beta_{#1,\tau}} 
\def\auttau#1{\beta_{#1}} 
\def\mass{\omega}

\title[]{Schwartz correspondence\\ for real motion groups in low dimensions}

\author{Francesca Astengo}
\address{Dipartimento di Matematica, Universit\`a di Genova, Via Dodecaneso 35, 16146 Genova, Italy} 
\email{{\tt astengo@dima.unige.it}}

\author{Bianca Di Blasio}
\address{Dipartimento di Matematica e Applicazioni, Universit\`a di Milano Bicocca, Via Cozzi 53, 20125 Milano, Italy } 
\email{{\tt bianca.diblasio@unimib.it}}

\author{Fulvio Ricci}
\address{Scuola Normale Superiore, Piazza dei Cavalieri
7, 56126 Pisa, Italy } 
\email{{\tt fulvio.ricci@sns.it}}



\begin{document}

\begin{abstract}

 For a Gelfand pair $(G,K)$ with $G$ a Lie group of polynomial growth and $K$ a compact subgroup, the {\it Schwartz correspondence} states that the spherical transform maps the bi-$K$-invariant Schwartz space ${\mathcal S}(K\backslash G/K)$ isomorphically onto the space ${\mathcal S}(\Sigma_{\mathcal D})$, where $\Sigma_{\mathcal D}$ is an embedded copy of the Gelfand spectrum in ${\mathbb R}^\ell$, canonically associated to a generating system ${\mathcal D}$ of $G$-invariant differential operators on $G/K$, and ${\mathcal S}(\Sigma_{\mathcal D})$ consists of restrictions to $\Sigma_{\mathcal D}$ of Schwartz functions on ${\mathbb R}^\ell$.

Schwartz correspondence is known to hold for a large variety of Gelfand pairs of polynomial growth. In this paper we prove that it holds for the strong Gelfand pair $(M_n,SO_n)$ with $n=3,4$.
The rather trivial case $n=2$ is included in previous work by the same authors. 
\end{abstract}

\maketitle

\baselineskip15pt

\section{Introduction}
Let $(G,K)$ be a Gelfand pair, with $G$ a connected Lie group and $K$ a compact subgroup of it. 
By definition, this means that the convolution algebra $L^1(K\backslash G/K)$ of bi-$K$-invariant integrable functions on $G$, i.e. satisfying
\be\label{bi-K}
f(k_1gk_2)=f(g)\ ,\qquad \forall\,g\in G\ ,\quad k_1,k_2\in K\ ,
\ee 
is commutative, or, equivalently, that the composition algebra $\bD(G/K)$ of $G$-invariant differential operators on $G/K$ is commutative.

 The Gelfand spectrum $\Sigma$ of $L^1(K\backslash G/K)$ is the space of bounded spherical functions on $G$ with the topology induced by the weak* topology on $L^\infty(K\backslash G/K)$. For each choice of a finite generating subset $\cD=\{D_1,\dots, D_\ell\}$ of $\bD(G/K)$, $\Sigma$ can be homeomorphically embedded into $\bC^\ell$, by assigning to each spherical function $\ph\in\Sigma$  the $\ell$-tuple $\xi=(\xi_1,\dots,\xi_\ell)$  if $D_j\ph=\xi_j\ph$ for $j=1,\dots,\ell$. The image $\Sigma_\cD$ of this embedding is a closed subset of $\bC^\ell$ \cite{FR}. 
 
Denoting by $\ph_\xi$ the spherical function associated with $\xi\in\Sigma_\cD$, the spherical transform $\cG f$ of $f\in  L^1(K\backslash G/K)$ can then be regarded as a map defined on $\Sigma_\cD$ by
$$
\cG f(\xi)=\int_Gf(x)\,\ph_\xi(x^{-1})\,dx\ .
$$

\medskip

If $G$ has polynomial volume growth (in which case we also say that the pair $(G,K)$ has polynomial growth) and the chosen generators $D_j\in\cD$ are  symmetric,
the eigenvalues are real, so that $\Sigma_\cD\subset\bR^\ell$. We refer to \cite{ADR3} for a presentation of Gelfand pairs of polynomial growth and the proofs of various preliminary results that will be used in this paper.

Gelfand pairs were originally introduced in the context of spherical analysis on symmetric spaces \cite{Ge}. The interest in pairs of polynomial growth, and in particular of the nilpotent Gelfand pairs mentioned below, is more recent \cite{BeJeRa1, BeJeRa2, KaRi, La1, La2, Ri, Vin, W, Y}.

We say that  {\it Schwartz correspondence} holds for a Gelfand pair $(G,K)$ of polynomial growth if the following property is satisfied:
\begin{enumerate}
\item[(S)] {\it The spherical transform maps the bi-$K$-invariant Schwartz space $\cS(K\backslash G/K)$ isomorphically onto the space $\cS(\Sigma_\cD)$ of restrictions to $\Sigma_\cD$ of Schwartz functions on~$\bR^\ell$.}
\end{enumerate}

This is an intrinsic property of the pair because it does not depend on the choice of the generating system $\cD$ \cite{ADR2, ADR3, FRY1}. 

The range of Gelfand pairs for which property (S) has been proved includes those with  $G$  compact (compact pairs) \cite{ADR3},  various families with $G=K\ltimes H$ and $H$ nilpotent (nilpotent pairs) \cite{ADR1, ADR2, FiR, FRY1, FRY2}, and those with $G=K\ltimes H$ and $K$ abelian  \cite[Thm. 8.1]{ADR3}.
No example in which condition (S) fails has been found so far.

An interesting class of Gelfand pairs for which Schwartz correspondence has not been established yet consists of the {\it strong Gelfand pairs}.

Denoting by ${\rm Int}(K)$ the group of inner automorphisms of $G$ induced by  elements of $K$ and calling $K$-central a function on $G$ that is ${\rm Int}(K)$-invariant,  the pair $(G,K)$ is called a  strong Gelfand pair if the algebra $L^1(G)^{{\rm Int}(K)}$ of $K$-central integrable functions is commutative.
The term ``strong'' is then justified by the obvious fact that bi-$K$-invariant functions are $K$-central.

It is easy to verify that $L^1(G)^{{\rm Int}(K)}\approx L^1\big(K\backslash (K\ltimes_{\rm Int} G)/K\big)$, so that  $(G,K)$ is strong Gelfand  if and only if $(K\ltimes_{\rm Int} G,K)$ is Gelfand.

According to a general structure of Gelfand pairs due to Vinberg 
\cite[Thm.~13.3.20]{W}, Gelfand pairs  $(G,K)$ must have $G=L\ltimes H$, with $H$~nilpotent of step $\le2$ and $L\supseteq K$ orbit-equivalent to $K$. The case of a strong Gelfand pairs corresponds to $L= K\ltimes_{\rm Int} K$.

A careful analysis of the classification in \cite{Y} shows that the non-compact strong Gelfand pairs $(G,K)$  with polynomial growth have $G=K\ltimes H$ and one of the following forms:
\begin{itemize} 
\item $K=SO_n$ and $H=\bR^n$, 
\item $K=U_n$ and $H=\bC^n$ or  the Heisenberg group $H_n$,
\item  direct products of the above.
\end{itemize}

Objects and notions given above for general Gelfand pairs can be reformulated in the following terms for a strong Gelfand pair $(G,K)$, cf. \cite[Sect. 2.1]{ADR3}:
\begin{itemize}
\item $\bD(G)^{{\rm Int}(K)}$ is the algebra of $K$-central differential operators on $G$,
\item the spherical functions are the $K$-central eigenfunction of all operators in  $\bD(G)^{{\rm Int}(K)}$,
normalized so that they take value $1$ at the identity.
\end{itemize} 

Then the Gelfand spectrum $\Sigma$ consists of the bounded spherical functions and,
given a finite generating system $\cD$ of $\bD(G)^{{\rm Int}(K)}$, the embedded copy $\Sigma_\cD$ of  $\Sigma$ is defined as the set of $({\rm dim}\cD)$-tuples of eigenvalues like in the general case, and the spherical transform is defined as in \eqref{bi-K}.
The formulation of Schwartz correspondence for strong Gelfand pairs with polynomial growth is the same as in (S), only with $\cS(K\backslash G/K)$ replaced by $\cS(G)^{{\rm Int}(K)}$.

The cases in which $K$ is abelian are already covered by the positive results in \cite{ADR3}.
The authors have recently proved property (S) for the strong pair $\big(M_2(\bC),U_2\big)$, where $M_2(\bC)=U_2\ltimes\bC^2$ is the complex motion group in two dimensions \cite{ADR4}. 

In this paper we consider the strong pairs $\big(M_n(\bR),SO_n\big)$ with $n=3,4$
and prove the following.
\begin{theorem} \label{main}
The strong Gelfand pairs $\big(M_n(\bR),SO_n\big)$, with $n=3,4$, satisfy property {\rm(S)},  i.e., the spherical transform maps $\cS(M_n(\bR))^{{\rm Int}(SO_n)}$  isomorphically onto $\cS(\Sigma_\cD)$.
\end{theorem}

\medskip

One important tool in the analysis of $K$-central functions for general strong Gelfand pairs is  decomposition into $K$-types.

Given $\tau\in\widehat K$,  let  $f_\tau$ be the $K$-type component of $f\in L^1(G)^{{\rm Int}(K)} $  defined by
\be\label{tipi}
f_\tau(g)=f*_K(d_\tau\overline{\chi_\tau})(g)=(d_\tau\overline{\chi_\tau} )*_Kf(g)=\int_K f(gk\inv)(d_\tau\overline{\chi_\tau})(k)\,dk\ ,
\ee
where $d_\tau$ and $\chi_\tau$ are character and dimension of $\tau$, respectively.
We  say that $f$ has $K$-type $\tau$ if $f_\tau=f$. The algebra  $L^1(G)^{{\rm Int}(K)}$ decomposes as the direct sum of the subalgebras $L^1(G)_\tau^{{\rm Int}(K)}$ of $K$-central functions of $K$-type $\tau$. Morever, for $\tau\ne\tau'$, 
\be\label{orthogonality}
L^1(G)_\tau^{{\rm Int}(K)}*L^1(G)_{\tau'}^{{\rm Int}(K)}=\{0\}\ ,
\ee
and, denoting by $\tau_0$ the trivial representation, $L^1(G)_{\tau_0}^{{\rm Int}(K)}=L^1(K\backslash G/K)$.

It follows that $(G,K)$ is a strong Gelfand pair if and only $(G,K,\tau)$ is a {\it commutative triple} (i.e., $L^1(G)_\tau^{{\rm Int}(K)}$ is commutative) for every $\tau\in\widehat K$~\cite{RS}. If this is the case, then $\Sigma$ is the disjoint union of the spectra $\Sigma^\tau$ of $L^1(G)_\tau^{{\rm Int}(K)}$, consisting of the bounded spherical function that are of type $\tau$.

We can then extend the notion of Schwartz correspondence to the commutative triple $(G,K,\tau)$ as follows.
\begin{enumerate}
\item[(${\rm S}_\tau$)] {\it The $\tau$-spherical transform maps  $\cS(G)^{{\rm Int}(K)}_\tau$ isomorphically onto the space $\cS(\Sigma^\tau_\cD)$,
where $\Sigma^\tau_\cD$ is the Gelfand spectrum of $L^1(G)_\tau^{{\rm Int}(K)}$.
}
\end{enumerate}

It is quite clear that, if the strong Gelfand pair $(G,K)$ satisfies property (S), then property (${\rm S}_\tau$) is satisfied for every $\tau\in\widehat K$. Conversely, if we assume that property (${\rm S}_\tau$) has been proved for every $\tau$, it is still necessary  to take into account the dependence  on $\tau$ of the Schwartz norm estimates involved. The following reduction principle of property (S) to single $K$-types has been used in~\cite{ADR4} for the strong Gelfand pair $\big(M_2(\bC),U_2\big)$ and will be used here to prove Theorem \ref{main}.

\begin{theorem}[{\cite[Prop. 5.2 and Thm. 7.1]{ADR3}}]\label{rendiconti}
Property {\rm  (S)} holds for a strong Gelfand pair $(G,K)$ of polynomial growth if and only if the following condition is satisfied:
\begin{enumerate}
\item[\rm(S')] given $f\in \cS(G)^{{\rm Int}(K)}$ and $N\in\bN$, for each $K$-type component $f_\tau$ of $f$, $\tau\in\widehat K$,   $\cG f_\tau$ admits a Schwartz extension $u_{\tau,N}$ such that $\|u_{\tau,N}\|_{(N)}$ is rapidly decaying in $\tau$.
\end{enumerate}
\end{theorem}

In the two cases considered here, $n=3,4$, the copy $\Sigma_\cD$ of the spectrum is embedded in $\bR^n$ and its subsets $\Sigma^\tau_\cD$ lie in distinct parallel 2-dimensional affine subspaces, so that, as long as we treat each $\tau$ separately, $\Sigma^\tau_\cD$ can be regarded as being embedded in $\bR^2$ and $\cG f_\tau$ as a function of two variables only. In this picture $\Sigma^\tau_\cD$ is the union of finitely many half-parabolas if $n=3$, and of finitely many half-lines if $n=4$, all exiting from the origin. The number of half-parabolas, or of half-lines, increases with the dimension of $\tau$. 

In order to obtain property (S') we prove the following two facts:
\begin{enumerate}[\rm(i)]
\item for general $f\in\cS\big(M_n(\bR)\big)^{{\rm Int}(SO_n)}$ and any $\tau\in\widehat{SO_n}$, it is possible to determine a formal two-variable power series $s_\tau(\xi_1,\xi_2)=\sum_{p,q}\frac{a^\tau_{p,q}}{p!q!}\xi_1^p\xi_2^q$  with the property that, for any $N\in\bN$, there are functions $h_{\tau,N}\in\cS(\bR^2)$ with rapidly decaying $N$-th Schwartz norm,  Taylor expansion at $0$ equal to $s_\tau$, and $\cG f_\tau-h_{\tau,N} $
vanishing of infinite order at~$0$
(Propositions~\ref{taylor} and~\ref{jet-extension});
\item an explicit extension formula for $\cG f_\tau$ exists, with the  decay in $\tau$ prescribed in (S'), under the assumption that $\cG f_\tau$ vanishes of infinite order at $0\in\Sigma_\cD^\tau$ (Proposition~\ref{p:nulla}).
\end{enumerate}

Once (i) and (ii) are verified, an application of  \cite[Prop.~4.2.1]{Mar1} gives the desired extensions $u_{\tau,N}$ of $\cG f_\tau$, satisfying condition~(S') (Proposition~\ref{S'}).
Following this approach, the major obstacle consists in proving the existence of the formal power series $s_\tau$ with the properties in (i), because the analysis of $\cG f_\tau$ near 0 only gives, at any degree, a certain number of linear relations among the unknown coefficients $a^\tau_{p,q}$. These relations lead to infinitely many, overdetermined in general, linear systems in the $a^\tau_{p,q}$ and the crucial problem is their solvability, a property apparently impossible to obtain by plain linear algebra.

Once solvability is proved the decay estimates in (i) can be proved in a relatively easy way.
In order to bypass this problem, we introduce an analytic argument to directly obtain $({\rm S}_\tau)$ for every~$\tau$, though not necessarily with the decay estimates required in~(S').

 This argument exploits the identification, via an isomorphism $A_\tau$, of $L^1\big(M_n(\bR)\big)_\tau^{{\rm Int}(SO_n)}$ with the algebra $L^1\big(\bR^n,{\rm End}(\spaziorp_\tau)\big)^{SO_n}$ of ${\rm End}(\spaziorp_\tau)$-valued, $SO_n$-equivariant, integrable functions on $\bR^n$, $\spaziorp_\tau$ being a representation space of $\tau$.
The map $A_\tau$ also gives a correspondence between the two sets of spherical functions and induces a homomorphisms from $\bD\big(\! M_n(\bR)\big)^{{\rm Int}(\grpSO_n)\!}$ onto 
$\big(\bD(\bR^n)\otimes {\rm End}(\spaziorp_\tau)\big)^{\grpSO_n\!}$, 
the algebra of ${\rm End}(\spaziorp_\tau)$-valued, $\grpSO_n$-equivariant differential operators on $\bR^n$.

For both $n=3$ and $n=4$, this last algebra admits a single generator $\mathbf D_\tau$ as a 
$\bD(\bR^n)^{\grpSO_n}$-module.
For $F\in \cS\big(\bR^n,{\rm End}(\spaziorp_\tau)\big)^{SO_n}$ it is then possible to obtain the representation formula
(Proposition~\ref{sommaF2})
\be\label{F=sum}
F=\sum_{j=0}^m\mathbf D_\tau^j g_j\ ,
\ee
 with $m$ depending on $\tau$ and the $g_j$ radial, scalar-valued, Schwartz functions on $\bR^n$. This reduces the problem to the existence of a Schwartz extension for the Hankel transform of each $g_j$, and this is a trivial consequence of the Whitney representation theorem for even functions on $\bR$. 
 
Though the general plan described above corresponds to that already used in~\cite{ADR4}, the detailed treatment presents here some new complications.
In principle, the same strategy of proof of Schwartz correspondence can be extended to real and complex motion pairs in higher dimensions, but the algebraic problems arising in Sections 4 and 5 become much more delicate. 

 \medskip

 Our paper is organized as follows.
In Section 2 we recall some preliminary facts about strong Gelfand pairs of polynomial growth and describe the decomposition into $K$-types.
   Moreover, specializing to the case where  $G=K\ltimes_\delta H$, we describe the isomorphism $A_\tau$ {between $\cS\big(G\big)^{{\rm Int}(K)}_\tau$ and  $\cS\big(H,{\rm End}(\spaziorp_\tau)\big)^{K}$. 
In Section 3 we fix $\tau\in \widehat{\grpSO_n}$ and, for any $n$,  we recall  from \cite{RS} some general facts about the algebra
 $\big(\bD(\bR^n)\otimes {\rm End}(\spaziorp_\tau)\big)^{\grpSO_n}$  and     
the Gelfand spectrum 
of $L^1\big(\bR^n,{\rm End}(\spaziorp_\tau)\big)^{\grpSO_n}$.
 Starting from Section 4, we limit ourselves to the cases $n=3,4$. First 
we choose  a priviliged generating  system 
of   $\big(\bD(\bR^n)\otimes {\rm End}(\spaziorp_\tau)\big)^{\grpSO_n}$
and 
spot the main differences between the cases $n=3$ and $n=4$. Then we fix notation in order to treat the two cases in a unified way as far as possible and  we determine  the Gelfand spectrum $\Sigma^\tau_\cD$.
In Section 5 we apply the results of the previous section to obtain the decomposition  formula~\eqref{F=sum}
 and we show  that property (${\rm S}_\tau$) holds  for each $\tau\in \widehat{\grpSO_n}$.
 In Section 6 we derive, for fixed $\tau$, the linear systems whose solutions   allow us to obtain the infinite order local approximations $\{h_{\tau, N}\}_N$ of $\cG_\tau f_\tau$ with rapid decay, according to (i) above.
At this stage, dealing with the linear systems,  we have to introduce new technical arguments with respect to \cite{ADR4}.
Finally we sum up all   our results proving that the Schwartz correspondence  holds for $\big(M_n(\bR),SO_n\big)$ with $n=3,4$.

\section{Spherical transforms for strong Gelfand pairs}\label{generale}

In this section we fix notation, recall basic facts and prove some statements in the general context of a strong Gelfand pair $(G,K)$ of polynomial growth. For details and unproven statements, we refer to \cite{RS, ADR3}.

The main result in the first subsection is the polynomial growth of the norms of  $\cG_\tau^{-1}$, the inverse of the $\tau$-spherical transform on $\cS(G)^{{\rm Int}(K)}_\tau$. 

In the second subsection we specialize   to the case where $G=K\ltimes_\delta H$.
In the analysis of functions of a given $K$-type, a matrix realization is often used, see \cite{RS} and the references therein. Specifically,
in Proposition~\ref{norme-f-Af}  we prove that $\cS(G)^{{\rm Int}(K)}_\tau$ is algebraically and  topologically isomorphic to  $\SmatrH$, the space of Schwartz  
 ${\rm End}(\spaziorp_\tau)$-valued equivariant functions  on $H$, i.e. satisfying \eqref{equivariance}.

\medskip
\subsection{The $\tau$-spherical transform and its inverse.} 
Let $(G,K) $  be a strong Gelfand pair.
We need notational conventions in order to distinguish among
the different Gelfand structures that intervene.

The plain symbols $\Sigma$ and $\cG$ denote, respectively, Gelfand spectrum and spherical  transform of $(G,K)$ as strong pair, i.e., the set of multiplicative linear functionals of the full algebra $L^1(G)^{{\rm Int}(K)}$ endowed with the weak$^*$ topology. 
The symbols $\Sigma^\tau$ and $\cG_\tau$ denote, respectively, Gelfand spectrum and spherical transform (or $\tau$-{\it spherical transform}) of 
the commutative triple $(G,K,\tau)$, i.e., of the  algebra $L^1(G)^{{\rm Int}(K)}_\tau$.
As usual, each multiplicative linear functional
 is realized by integration with a bounded spherical function. 
It is a fact \cite{RS} that every bounded spherical function is of $K$-type $\tau$ for some, hence unique, 
$\tau\in\widehat K$. Hence $\Sigma$  is the disjoint union of the spectra  $\Sigma^\tau$, consisting of the bounded spherical functions  of type~$\tau$ (or $\tau$-{\it spherical functions}). 
In particular, $\Sigma^{\tau_0}$ and $\cG_{\tau_0}$ are  Gelfand spectrum and spherical transform of the (ordinary) Gelfand pair $(G,K)$.
By \eqref{orthogonality}, the spherical transform $\cG f$ of $f\in L^1(G)^{{\rm Int}(K)}$
 restricts to  $\Sigma^\tau$ as the $\tau$-spherical transform $\cG_\tau f_\tau$
 of its $\tau$-component.
\medskip

If $G$ 
is connected and
denoting by $\bD(G)^{\rm{Int}(K)}$ the algebra of left-invariant differential operators on $G$ that are also ${\rm Int}(K)$-invariant, the bounded spherical functions are the elements of $C^\infty(G)^{{\rm Int}(K)}$ that are eigenfunctions of all elements of $\bD(G)^{\rm{Int}(K)}$, normalized to attain value 1 at the identity.

By the general facts about Gelfand pairs mentioned in the introduction, it follows that, if $G$ has also polynomial growth, the choice of a finite system 
$\cD=\{D_1,\dots,D_\ell\}$ of symmetric
 generators of $\bD(G)^{\rm{Int}(K)}$ determines a homeomorphism $\iota_\cD$ of $\Sigma$ with a closed subset  $\Sigma_\cD$ of  $\bR^{\ell}$.

We shall always assume that $G$ is connected with polynomial growth.

It is convenient to choose $\cD$ of the form $\cD=\cD'\cup\cD''$, where $\cD'=\{D_1,\dots,D_{\ell'}\}$ generates the center of $\bD(K)$. Then the homeomorphism $\iota_\cD=(\iota_\cD',\iota_\cD'')$ maps $\ph\in\Sigma$ to a point 
$$
\xi=(\xi',\xi'')=\big(\iota'_\cD(\ph),\iota''_\cD(\ph)\big)\in\bR^{\ell'}\times\bR^{\ell''}\ ,\qquad(\ell''=\ell-\ell')\ .
$$
When $\ph$ is $\tau$-spherical, the coordinates of $\xi'_\tau=\iota'_\cD(\ph)$ are the eigenvalues of  $\overline{\chi_\tau}$  under 
$\{D_1,\dots,D_{\ell'}\}$ and 
$$
\Sigma_\cD=\bigcup_{\tau\in\widehat K}\{\xi'_\tau\}\times \Sigma^\tau_\cD
$$
where
$$
\Sigma^\tau_\cD=\big\{\iota''_\cD(\ph):\ph\in\Sigma^\tau\big\}\subset\bR^{\ell''}
$$
is homeomorphic to $\Sigma^\tau$. Note also that rapid decay in $\tau$ can be measured in terms of powers of $|\xi'_\tau|$
(see~\cite[formula~(7.2)]{ADR3}).

With a language abuse the spherical transform 
of a function in $\cS(G)^{{\rm Int}(K)}$ will be considered as a function on $\Sigma_\cD\subset\bR^\ell$, and the $\tau$-spherical transform 
of a function in $\cS(G)^{{\rm Int}(K)}_\tau$ will be considered as a function on 
$\Sigma^\tau_\cD\subset\bR^{\ell''}$.
In formulae, for every $f\in \cS(G)^{{\rm Int}(K)}$,
\[
\begin{aligned}
&\cG f (\xi)=\int_G f(x)\, \ph(x^{-1})\, dx \qquad 
\forall  \xi=\iota_\cD(\ph)\in\Sigma_\cD 
\\
&\cG_\tau f_\tau(\xi'') 
=\cG f(\xi'_\tau,\xi'')\qquad 
\forall \xi'' =\iota''_\cD(\ph)\in \Sigma^\tau_\cD\ .
\end{aligned}
\]
In particular, when $f$ is of type $\tau$,
\[
\cG f (\xi'_{\sigma},\xi'')=
\begin{cases}
\cG_\tau f (\xi'') & \text{if }\sigma=\tau
\\
0 & \text{if } \sigma\neq \tau\ .
\end{cases}
\]
\medskip
For a fixed compact, symmetric neighborhood $U$ of the identity $e$ of $G$, define 
$$
|x|=\begin{cases} 0&\text{ if }x=e\\ \min\{n:x\in U^n\}&\text{ if }x\neq e\ .\end{cases}
$$
For our convenience 
we will use the following  Schwartz norms.
For $N\in \mathbb N$,
 \beas
&\|u\|_{(N)}=\max_{|\al |\le N}\|\big(1+|\cdot |\big)^N\partial^\al u \|_{L^2} \qquad \forall u\in\cS(\bR^\ell)
\\
&\|f\|_{(N)}=\max_{|\al|\le N}\|\big(1+|\cdot|\big)^N\, D^\al f\|_{L^2} \qquad \forall f \in \cS(G)^{{\rm Int}(K)}\ ,
\eeas
where for every multiindex $\al=(\al_1,\dots,\al_\ell)$
$$
\partial^\al=\partial_{\xi_1}^{\al_1}\partial_{\xi_2}^{\al_1}\cdots \partial_{\xi_\ell}^{\al_\ell} 
\qquad\text{and}\qquad
D^\al=D_1^{\al_1}D_2^{\al_2}\cdots D_\ell^{\al_\ell} .
$$ 
When $C$ is a closed set of $\bR^\ell$,  $\cS(C)$ denotes the   space  of restrictions to $C$ of Schwartz functions on~$\bR^\ell$, that is
$$ 
\cS(C)=\big\{u_{|_{C}}\,:\,u\in\cS(\bR^\ell)\big\} \cong \cS(\bR^\ell)/\{u\,:\,u_{|_{C}}=0\}\ ,
$$
endowed with the quotient topology, induced by the norms i.e.
\[
\|\psi\|_{(N)}=\inf\{\|u\|_{(N)}\, :\, u_{|_{C}}=\psi\}
\qquad \psi\in \cS(C)\quad N\in \bN \ .
\]

By the Plancherel-Godement Theorem~\cite[p.~193]{W}
 the spherical transform $\cG$ is a unitary operator  from $L^2(G)^{{\rm Int}(K)}$ onto $L^2(\Sigma_\cD, \beta)$, where $\beta$ is the Plancherel  measure on $\Sigma_\cD$. 
 Since the symmetric operators $D_j$ 
are essentially self-adjoint~\cite{Nels-Steins}, the set $\Sigma_\cD\subset \bR^\ell$ is their joint $L^2(G)^{{\rm Int}(K)}$-spectrum (see \cite[Proposition~4.4]{ADR3}).
  In spectral analytic terms, $\cG\inv \psi$ is the convolution kernel of $\psi(D_1,\dots,D_\ell)$.

We can then rephrase a multiplier theorem by 
A.~Martini~\cite[Prop. 4.2.1]{Mar1} 
in the following way.

\begin{proposition}\label{nucleo}
The map $\cG\inv$ is continuous from $ \cS(\Sigma_\cD)$ in $\cS(G)^{{\rm Int}(K)}$, i.e., for every $N\in \bN$ there exist $M=M(N)$ and $C_N$ such that 
$$
\|\cG^{-1}u \|_{(N)} \leq C_N\, \|u\|_{(M)}\qquad \forall u \in  \cS(\Sigma_\cD).
$$
\end{proposition}

The use of this proposition is twofold. First, Proposition~\ref{nucleo}, combined with the open mapping theorem for Fr\'echet spaces,
allows us to disregard continuity of $\cG$. In other words, in order to prove that property (S) holds for the strong Gelfand pair $(G,K)$, we only need to prove that {$\cG^{-1}$ is surjective, i.e., for every $f\in \cS(G)^{\text{Int}(K)}$, the spherical transform $\cG f$ extends to a Schwartz function on $\bR^\ell$. 
This will be proved by reduction to $K$--types according to the   criterion in Theorem~\ref{rendiconti}.

The second application of Proposition~\ref{nucleo} consists in establishing an analogous result for a fixed $K$--type $\tau$. This will be used in the proof of our main result {(Proposition~\ref{S'})} as a tool
 in the construction of the extension with norms of rapid decay in $\tau$ as required in Theorem~\ref{rendiconti}.

\begin{corollary}\label{nucleo-tipo-tau} Let $\tau$ be in $\widehat K$.
Every $v\in  \cS(\Sigma^\tau_\cD)$ is the spherical transform of a unique function $\cG^{-1}_\tau v\in  \cS(G)_\tau^{{\rm Int}(K)}$.
Moreover for every $N\in \bN$ there exist $M=M(N)$, $M'=M'(N)$ and $C_N$ such that 
$$
\|\cG^{-1}_\tau v \|_{(N)} \leq C_N\, (1+|\xi'_\tau|)^{M'}\,\, \|v\|_{(M)}\qquad \forall v \in  \cS(\Sigma^\tau_\cD).
$$
\end{corollary}

\begin{proof}  The starting point is the existence, for each $\tau\in \widehat K$, of a Schwartz function
$u_\tau\in \cS(\bR^{\ell'})$ such that $u_\tau(\xi'_\tau)=1$, $u_\tau(\xi'_\sigma)=0$
if $\sigma\in\widehat K$, $\sigma\ne\tau$ and
for every $N$ there exist $ N'=N'(N)$ and $C_N$ such that
\[
\|u_\tau\|_{(N)}\leq C_N\, (1+|\xi'_\tau|)^{N'}.
\]
This fact has been proved in~\cite[Lemma 7.2]{ADR3}.

Suppose now we are given $v\in  \cS(\Sigma^\tau_\cD)$ and let   $v_\ast\in\cS(\bR^{\ell''})$ be an extension of it. Define 
a Schwartz function $ u_\ast$ on  $\cS(\bR^\ell)$ by 
$$
 u_\ast(\xi',\xi'')=u_\tau(\xi')\,v_\ast(\xi'')\qquad \forall  (\xi',\xi'')\in  \bR^\ell
$$
and denote by $u$ its restriction to $\Sigma_\cD$. 
By Proposition~\ref{nucleo}
the function $f=\cG^{-1}  u$ belongs to $  \cS(G)^{{\rm Int}(K)}$ and
 for every $N\in \bN$ there exist $M=M(N)$, $M'=M'(N)$ and $C_N$ such that 
$$
\|f \|_{(N)} \leq C_N\, \| u\|_{(M)} \leq C_N\, (1+|\xi'_\tau|)^{N'} \|v\|_{(M)} \ .
$$
Finally, 
\[
\cG f(\xi'_\sigma,\xi'')=u(\xi'_\sigma,\xi'')=u_\tau(\xi'_\sigma)\,v(\xi'') =0\qquad \sigma\neq \tau
\]
which implies that $f$ is of type $\tau$ and
\[
v(\xi'')= u(\xi'_\tau,\xi'')=\cG f(\xi'_\tau,\xi'')=\cG_\tau f(\xi'')\qquad \forall \xi''\in \Sigma^\tau_\cD\ .
\]
Uniqueness of $v$ is obvious, because $L^1(G)^{{\rm Int}(K)}_\tau$ is semisimple.
\end{proof}

\bigskip
\subsection{${\rm End}(\spaziorp_\tau)$-valued Schwartz functions}

\quad 
For the purposes of this paper, we will introduce ${\rm End}(\spaziorp_\tau)$-valued functions only in the case of the semidirect product $G=K\ltimes_\delta H$, with $K$ compact and acting on the nilpotent group $H$ by $\delta$.

We denote by $\spaziorp_\tau$ the representation space of $\tau$, by $d_\tau$ its dimension and by ${\rm End}(\spaziorp_\tau)$ the algebra of linear endomorphisms of $\spaziorp_\tau$. The space of integrable 
 ${\rm End}(\spaziorp_\tau)$-valued functions on $H$ is an algebra with convolution defined by
$$
F_1*F_2(h)=\int_H F_2(z\inv h)F_1(z)\,dz
$$
and norm $\|F\|_1=\int_H\|F(h)\|_{\rm op}\, dh$.

We say that an ${\rm End}(\spaziorp_\tau)$-valued function~$F$ on $H$ is $K$-equivariant
if it satisfies 
\be\label{equivariance}
F(\delta(k)h)=\tau(k)F(h)\tau(k\inv)\ ,\qquad\forall\,k\in K\ ,
\ee
and denote by $\LmatrH$, and similarly for other function spaces, the subspace of integrable $K$-equivariant functions.

The two maps
$$
\begin{array}{rcrcrcl}A_{\tau}&:&f(k,h)&\longmapsto &F(h)&=&\int_Kf(k,h)\tau(k)\,dk\\ 
A_{\tau}\inv&:&F(h)&\longmapsto &f(k,h)&=&d_\tau\,\tr\big(\tau(k\inv)F(h)\big)
\end{array}
$$
establish a one-to-one correspondence between $K$-central functions $f$ on $G$ of $K$-type $\tau$ and $K$-equivariant ${\rm End}(\spaziorp_\tau)$-valued functions on $H$.

In particular, $A_{\tau}$ is an isomorphism of algebras from $L^1(G)^{{\rm Int}(K)}_{\tau}$ onto 
$\LmatrH$ and $\sqrt{d_\tau}A_{\tau}$ is unitary from $L^2(G)^{{\rm Int}(K)}_{\tau}$ onto 
$\LduematrH$, where
$$
\|F\|_2^2=\int_{H}\big\|F(h)\big\|^2_{{\rm HS}}\,dh
\qquad \forall F\in \LduematrH \ 
$$
and $\|\cdot\|_{{\rm HS}}$ is the Hilbert-Schmidt norm.

Via the map $A_\tau$, we can then redefine the ingredients of the spherical analysis of $K$-central functions of type $\tau$ on $G$ in terms of the above ${\rm End}(\spaziorp_\tau)$--valued model. 

\medskip
\centerline{
\begin{tikzpicture}[node distance=6.5cm, auto]
\node (A) {$L^1(G)^{{\rm Int}(K)}_\tau$};
\node (B) [right of=A] {$C(\Sigma_\tau)$};
\node (C) [node distance=3cm, below of=A] {$\LmatrH$};
\draw[->] (A) to node [swap] {$\cG_\tau$} (B);
\draw[->] (A) to node [swap] {$A_\tau$} (C);
\draw[->] (C) to node [swap] {$\cG_\tau^\sharp={\cG_{\tau}\circ { A_\tau^{-1}}}$} (B);
\end{tikzpicture}}
\medskip

This will be explained in detail in the next section in the case of the euclidean motion group,
where we explicitly determine the ${\rm End}(\spaziorp_\tau)$--valued spherical functions, their eigenvalues and we deduce, via $A_\tau$, the $\tau$-spherical ones.
We denote by $\big(\bD(H)\otimes {\rm End}(\spaziorp_\tau)\big)^K$
the algebra of ``${\rm End}(\spaziorp_\tau)$--valued'' differential operators on $H$ which commute with translations and with the action of $K$ on smooth ${\rm End}(\spaziorp_\tau)$--valued functions $F$ given by
$$
k:F\longmapsto 
\tau(k)F(\delta(k\inv) h)\tau(k\inv)\ .
$$
Conjugation by $A_{\tau}$ is a homomorphism of $\bD(G)^{{\rm Int}(K)}$   onto $\big(\bD(H)\otimes {\rm End}(\spaziorp_\tau)\big)^{K}$ and its kernel consists of the operators which vanish on functions of $K$-type $\tau$.

The following lemma allows us to disregard the operators in $\cD'$,  the center of  $\bD(K)$, in the evaluation 
of the Schwartz norms of functions in $\SmatrH$.

\begin{lemma} \label{operatoriDscalari}  Let $\tau\in \widehat K$  and $D_j$ be in $\cD'$.
Then $A_\tau D_j A_\tau^{-1}$ is the scalar multiplication operator by $\xi_j$ 
so that for every $F$ in $C^\infty\cap\LduematrH$
\[
\|A_\tau D A_\tau^{-1} F(h)\|_{\rm HS} \leq |\xi'_\tau|\, \|F(h)\|_{\rm HS}\qquad  \forall h\in H,\ \forall D\in \cD'.
\]
\end{lemma}

\begin{proof}
Let $\xi'_\tau=(\xi_1,\ldots,\xi_{\ell'})$, where $\xi_j$ is the eigenvalue   of $\overline\chi_\tau$ under $D_j\in \cD'$.  
Given $F$ in $C^\infty\cap\LduematrH$, the function  $f =A_\tau\inv F$
  is in $C^\infty\cap L^2(G)^{{\rm Int}(K)}_\tau$, so that by~\eqref{tipi}
  $$
D_j f=D_j (f*_K(d_\tau\overline{\chi_\tau})) =f*_K(d_\tau D_j\overline{\chi_\tau})=\xi_j\,f*_K(d_\tau\overline{\chi_\tau})=\xi_j f.
$$
Therefore   applying the isomorphism $A_\tau$ we obtain
\[
A_\tau D_j A_\tau\inv F=A_\tau D_j f=\xi_j\, A_\tau f=\xi_j\, F
\]
and the thesis follows.
\end{proof}

The following proposition establishes that $A_\tau$ is a topological isomorphisms from $\cS(G)^{{\rm Int}(K)}_\tau$ onto $\SmatrH$.

As $N$-order Schwartz norm of a function $F$ in $\SmatrH$ 
we take
\[
\|
F\|_{(N)}=\max_{  |\beta|\leq N}\|(1+|\cdot|)^N\,(D^\tau)^\beta F\|_{L^2}
\] 
where for every multiindex $\beta=(\beta_1,\dots,\beta_{\ell''})$ 
$$
(D^\tau)^\beta= A_\tau\, D_{\ell'+1}^{\beta_{1}}\cdots  D_{\ell}^{\beta_{\ell''}}\, A_\tau\inv.
$$

\begin{proposition}\label{norme-f-Af} Let $\tau\in \widehat K$ and $N$ in $\bN$. The following estimates hold 
\[
 \sqrt{d_\tau}\,\|A_\tau f\| _{(N)}\leq   \|f\|_{(N)}\leq (1+|\xi'_\tau|)^N\, \sqrt{d_\tau}\, \|A_\tau f\|_{(N)} \qquad \forall f\in \cS(G)^{{\rm Int}(K)}_{\tau } \ .
\]
\end{proposition}

\begin{proof} Let $f\in \cS(G)^{{\rm Int}(K)}_{\tau }$. 
Since $\sqrt{d_\tau}A_\tau$ is an isometry on $L^2$-spaces,
 we have
\beas
\|A_\tau f\|_{(N)}&=\max_{  |\beta|\leq N}\|(1+|\cdot |)^N\,(D^\tau)^\beta A_\tau f\|_{L^2}
\\
&=\max_{ |\beta|\leq N}\|A_\tau (1+|\cdot|)^N\,(D'')^\beta f\|_{L^2}
\\
&=\frac{1}{\sqrt{d_\tau}} \max_{  |\beta|\leq M}\|(1+|\cdot|)^N\,(D'')^\beta f\|_{L^2}
\\
&\leq \frac{1}{\sqrt{d_\tau}} \|f\|_{(N)}\ ,
\eeas
where $(D'')^\beta=(D_{\ell'+1})^{\beta_{1}}\cdots  (D_{\ell})^{\beta_{\ell''}}$. 
Conversely, using Lemma~\ref{operatoriDscalari} with $F=A_\tau f$
\begin{align*}
\|f\|_{(N)}&=\max_{  |\beta|\leq N}\|(1+|\cdot|)^N\,D^\beta A_\tau\inv F\|_{L^2}
\\
&=\max_{  |\beta|\leq N}\|A_\tau\inv (1+|\cdot|)^N\,A_\tau D^\beta A_\tau\inv F\|_{L^2}
\\
&\leq {\sqrt{d_\tau}}\, (1+|\xi'_\tau|)^N\,\max_{  |\beta|\leq N}\|(1+|\cdot|)^N\,(D^\tau)^\beta F\|_{L^2}
\\
&= {\sqrt{d_\tau}}\, (1+|\xi'_\tau|)^N\, \|F\|_{(N)}\ .
\qedhere
\end{align*}
\end{proof}
  
  \begin{corollary}\label{isoAtau}
  The map $A_\tau$ is a topological isomorphism from $\cS(G)^{{\rm Int}(K)}_\tau\!$ onto $\SmatrH$.
  \end{corollary}
  
\section{$\tau$-spherical analysis on the euclidean motion group $\grpSO_n\ltimes\bR^n$}
\bigskip

In this section we introduce the euclidean motion group $G=\grpSO_n\ltimes_\delta\bR^n$ for general $n$. We recall some basic facts about the algebra $\big(\bD(\bR^n)\otimes {\rm End}(\spaziorp_\tau)\big)^{K}$ and  
the Gelfand spectrum $\Sigma^\tau$ of $L^1\big(\bR^n,{\rm End}(\spaziorp_\tau)\big)^{K}$.

\subsection{The euclidean motion group} The euclidean motion group is the semidirect product $G=\grpSO_n\ltimes_\delta\bR^n$ where the action $\delta$ 
of $K=\grpSO_n$ on $\bR^n$
is the natural one, i.e., $\delta(k)y=k\, y$.
We write its elements as pairs $(k,y)\in \grpSO_n\times \bR^n$ with product law
$$
(k,y)(k',y')=(kk',y+k\,y')\ 
$$ 
so that the action of $K$ on $G$ by inner automorphisms is given by 
$$
k\,\,:\,\, (k',y')\longrightarrow (k\,k'\, k^{-1}, k\, y')\ .
$$
It is well known that the algebra $L^1(G)^{{\rm Int}(K)}$ is commutative, i.e.,
the pair $(G,K)$ 
is a strong Gelfand pair~\cite{RS,Vin,Y}.

\subsection{Invariant differential operators}\label{sect:invdiffop}
Differential operators in $\bD(G)^{\rm{Int}(K)}$   are obtained from ${\rm Ad}(K)$--invariant polynomials 
 on the Lie algebra $\fg$ of $G$ via the symmetrization procedure described in~\cite[formula (2.4)]{RS}.

Indeed, let $p$ be in $\cP(\fg)^{{\rm Ad}(K)}$ and note that $\fg$ can be identified with $\fk\oplus \bR^n$, where $\fk=\mathfrak{so}_n$, so that $p$ satisfies
 $$
p(X,y)=p(Ad(k)(X,y))=p(kXk\inv,ky) \qquad \forall k\in K,\quad (X,y)\in \fk\oplus\bR^n\ .
$$
Then,
using the symmetrization operator $\la$ on the universal enveloping algebra $\fU(\fk)$ \cite[p.~180]{Vara}
 and Fourier transform in $y$, we construct
the corresponding operator $D_p\in \bD(G)^{\rm{Int}(K)}$. 
Explicitly, for the choice of an orthonormal basis $\{X_1,\ldots,X_{n(n-1)/2}\}$ of $\mathfrak{so}_n$
we write $X=\sum_\ell x_\ell \, X_\ell$ and if
\be\label{e:forma-polinomio}
p(X,y)=\sum_{\al,\beta}c_{\al,\beta}\,x^\al y^\beta \ ,
\ee
we define 
\be\label{D_p}
D_p= \sum_{\al,\beta}c_{\al,\beta}\la((ix)^\al) \,(-i\de_y)^\beta\ .
\ee
}

The  coefficient $i$ in both factors has been introduced in order to map real polynomials into symmetric operators.

For the reader's convenience we recall that 
\[
(\lambda((ix)^\al)f)(g)=i^{|\al|}\, \partial_t^\alpha f(g\exp(t_1X_1+\cdots+t_mX_m))|_{t=(t_1,\ldots,t_m)=0}
\]
with $m=n(n-1)/2$, and that we choose the following normalization of the Fourier transform of an integrable function $f$ on $\bR^n$
\[
\mathcal F f(\eta)=\widehat f (\eta)=\int_{\bR^n}f(y)\, e^{-iy\cdot \eta}\,dy.
\]

Explicit
families of generating polynomials of $\cP(\fg)^{{\rm Ad}(K)}$ can be found in   [FRY1, Thm. 7.5(1)].
It is convenient for us to choose a set of generators 
$\{p_1, \dots p_{\ell'}, q_1,\ldots, q_{\ell''}\}$, where  the  polynomials $p_j$    depend only on  the variables in $\fk$,  the polynomial $q_1=|y|^2$   depends only on the $\bR^n$ variables and the polynomials $q_j$ with $j>1$ may depend on both groups of variables.

Calling $D_{p_j}$, $D_{q_j}$ the symmetrizations  of the above polynomials,
the operator $D_{q_1}$ is (minus) the laplacian on $\bR^n$, i.e., 
\[
D_{q_1}=\Delta=-\sum_{j=1}^n \partial_j^2
\]
 and
the operators $D_{p_j}$ form the set $\cD'$ in Section~\ref{generale} so that their eigenvalues on each spherical function
  depend only on its type.

\subsection{$K$-type subalgebras and ${\rm End}(\spaziorp_\tau)$-valued spherical analysis}
We now focus on the commutative triple $(G,K,\tau)$ with $\tau\in \widehat K$
and understand the structure of ${\rm End}(\spaziorp_\tau)$--valued functions $F$ on $\bR^n$ 
satisfying the $K$-equivariance condition \eqref{equivariance}, where $\delta(k)y=k\,y$.

This condition implies that such a function $F$ is uniquely determined by its values for $y=r\base$ where $\base=(1,0,\dots,0)$ and $r\ge0$, and that
$F(r\base)$ commutes with $\tau(k)$ for all $k$ in the stabilizer $K_\base$ of $\base$, i.e., $k=\begin{bmatrix}1&0\\0&k'\end{bmatrix}$ with $k'\in \grpSO_{n-1}$. 

Since $(\grpSO_n,\grpSO_{n-1})$ is also a strong Gelfand pair, the restriction of $\tau$ to $K_\base$ decomposes without multiplicities,
\be\label{restriction}
\tau_{|_{SO_{n-1}}}=\sigma_0+\cdots+\sigma_\finetau\ ,\qquad \spaziorp_\tau=\spaziorpU_{\sigma_0}\oplus\cdots\oplus \spaziorpU_{\sigma_\finetau}\ ,
\ee
where $\spaziorpU_{\sigma_j}$ is the representation space of $\sigma_j$ and the $\sigma_j$ are mutually inequivalent . Then, by Schur's lemma,
\be\label{diagonal}
 F(ro)=\sum_{s=1}^\finetau b_{s,\tau}(F,r)P_{s,\tau} \ ,\qquad b_{s,\tau}(F,r)\in\bC\ ,
 \ee
where $P_{s,\tau}$
 is the orthogonal projection of $\spaziorp_\tau$ onto $\spaziorpU_{\sigma_s}$, so that 
 $$
 b_{s,\tau}(F,r)=\frac1{d_{\sigma_s}}\tr\big(P_{s,\tau}     F(ro) \big),
 $$
 where $d_{\sigma_s}=\mathrm{dim}\,\spaziorpU_{\sigma_s}$.

Spherical functions have been determined in \cite[Sect.~7,~11]{RS} and we recall them in the following proposition.
The Fourier transform of ${\rm End}(\spaziorp_\tau)$-valued functions is
defined componentwise.

\begin{proposition}\label{sferichematr}
The Gelfand spectrum $\Sigma^\tau$ of $L^1\big(\bR^n,{\rm End}(\spaziorp_\tau)\big)^{K}$ consists of the multiplicative functionals 
$$
F\longmapsto b_{s,\tau}(\widehat F,\ro)
=\int_{\bR^n}\frac1{d_\tau}\tr\big(F(y)\Phi^\tau_{\ro,s}(-y)\big)\,dy\ ,
$$
where $\Phi^\tau_{\ro,s}$ is the ${\rm End}(\spaziorp_\tau)$-valued spherical function
\be\label{Phi}
\Phi^\tau_{\ro,s}(y)=\frac{d_\tau}{d_{\sigma_s}} \int_{K} e^{i\ro\base\cdot(ky)}\tau (k\inv)P_{s,\tau}\,\tau (k)\,dk\ 
\ee 
with $d_\tau= \mathrm{dim}\,\spaziorp_\tau$ and $d_{\sigma_s}=\mathrm{dim}\,\spaziorpU_{\sigma_s}$, $s=0,\ldots,\finetau$.
\end{proposition}

\begin{proof}
For the reader's convenience we explicit here  the link between the two formulas in the statement and we refer to~\cite{RS} for the other details.
By~\eqref{equivariance}, we have
\begin{align*}
b_{s,\tau}(\widehat F,\ro)
&=\frac1{d_{\sigma_s}}\tr\big(P_{s,\tau}  \widehat  F(\ro o) \big)
\\
&=\frac1{d_{\sigma_s}}\int_{\bR^n}e^{-i \ro\base\cdot y}\,\tr \big(P_{s,\tau} F(y)\big)\,dy\  
\\
&=\frac1{d_{\sigma_s}}\int_K\int_{\bR^n}e^{-i \ro\base\cdot(ky)}\,\tr \big(P_{s,\tau} F(ky)\big)\,dy\,dk\ 
\\
&=\frac1{d_{\sigma_s}}\int_K\int_{\bR^n}e^{-i \ro\base\cdot(ky)}\,\tr \big(
\tau (k\inv)P_{s,\tau}\,\tau (k) F(y)\big)\,dy\,dk\ 
\\
&
=\int_{\bR^n}\frac1{d_\tau}\tr\big(F(y)\Phi^\tau_{\ro,s}(-y)\big)\,dy\ .
\qedhere
\end{align*}
\end{proof}

\medskip
Via the map $A_\tau$ we infer that the multiplicative functionals  of $L^1(G)^{{\rm Int}(K)}_\tau$ 
are given by  
$$
f\longmapsto \int_{K\ltimes \bR^n}f(k,y)\ph^\tau_{\ro,s}\big((k,y)\inv\big)\,dk\,dy\ ,
$$
where the functions
$$
\ph^\tau_{\ro,s}(k,y)=\frac1{d^2_\tau}A_\tau\inv\Phi^\tau_{\ro,s}(k,y)=
\frac1{d_\tau}\,\tr\big(\tau(k)\inv \Phi^\tau_{\ro,s}(y)\big)\ ,
$$
are the $\tau$-spherical functions for the Gelfand pair $(K\ltimes_{\rm Int} G,K)$. Constants are arranged 
in such a way that $\Phi^\tau_{\ro,s}(0)$ is the identity operator and $\ph^\tau_{\ro,s}(e,0)=1$.

\subsection{Differential operators in $\big(\bD(\bR^n)\otimes {\rm End}(\spaziorp_\tau)\big)^{K}$}

\quad
Suppose that $p$ is in $\cP(\fg)^{{\rm Ad}(K)}$.
In Section~\ref{sect:invdiffop} we developed a rule
that generates the differential operator $D_p$ in $\bD(G)^{\rm{Int}(K)}$ (see formulas~\eqref{e:forma-polinomio} and~\eqref{D_p}).

We can then obtain the differential operator $D_p^\tau$ in $\big(\bD(\bR^n)\otimes {\rm End}(\spaziorp_\tau)\big)^{K}$
simply conjugating by $A_\tau$.

The purpose of the following lemma is to relate conjugation by $A_\tau$ to the symmetrization procedure, 
specializing~\cite[Corollary~2.3]{RS} to our case.

\begin{lemma}\label{lem:simmetrizzaz-vettoriale} Suppose that $p$ is as in \eqref{e:forma-polinomio}. Then for every $F$ in $\Smatr$ 
$$
D^\tau_pF(y)=\sum_{\al\beta}c_{\al\beta} \,(-i\partial_y)^\beta F(y)\,  d\tau(\lambda((-ix)^\al))
$$
so that
\[
\widehat{D^\tau_pF}(\eta)=\widehat{F}(\eta)\,\sum_{\al\beta}c_{\al\beta} \,\eta^\beta   d\tau(\lambda((-ix)^\al))
\]
\end{lemma}

\begin{proof} Let $f_\tau=A_\tau^{-1}F$. Then with $m=n(n-1)/2$
\begin{align*}
D^\tau_pF(y)&=A_\tau D_pf_\tau(y)
\\
&=\int_K (D_pf_\tau)(k,y)\, \tau(k)\, dk
\\
&=\sum_{\al\beta}c_{\al\beta} \, i^{|\al|}\,\int_K (\lambda(x^\al)(-i\partial_y)^\beta f_\tau)(k,y)\, \tau(k)\, dk
\\
&=\sum_{\al\beta}c_{\al\beta} \, i^{|\al|}\,\int_K (-i\partial_y)^\beta \, \partial_t^\alpha
f_\tau(k\exp(t_1X_1+\cdots+t_mX_m),y)|_{t=(t_1,\ldots,t_m)=0}\, \tau(k)\, dk
\\
&=\sum_{\al\beta}c_{\al\beta} \, i^{|\al|}\,(-i\partial_y)^\beta \, \partial_t^\alpha\!\left.
\int_K f_\tau(k\exp(t_1X_1+\cdots+t_mX_m),y)\, \tau(k)\, dk
\right|_{t=(t_1,\ldots,t_m)=0}
\\
&=\sum_{\al\beta}c_{\al\beta} \, i^{|\al|}(-i\partial_y)^\beta \left. \partial_t^\alpha\!
\int_K f_\tau(k',y)\, \tau(k')\tau(\exp(-(t_1X_1+\cdots+t_mX_m))\, dk'
\right|_{t=(t_1,\ldots,t_m)=0}
\\
&=\sum_{\al\beta}c_{\al\beta} \, i^{|\al|}(-i\partial_y)^\beta\! \int_K f_\tau(k',y)\,\tau(k')\, dk'\, 
\partial_t^\alpha\tau(\exp(-(t_1X_1+\cdots+t_mX_m))|_{t=(t_1,\ldots,t_m)=0}
\\
&=\sum_{\al\beta}c_{\al\beta} \,(-i\partial_y)^\beta F(y)\, d\tau(\lambda((-ix)^\al)) \ .
\qedhere
\end{align*}
\end{proof}

\section{Equivariant polynomials and the Gelfand spectrum $\Sigma^\tau_\cD$ when $n=3,4$}

From now on we will consider the cases where $n=3,4$.  
In this section we will compute, for any fixed $\tau$ in $\hat K$, the eigenvalues $\xi'_\tau$. 
Next we will consider the representation $\tilde\tau$ on 
$\mathrm{End}(\spaziorp_\tau)$ given by
\begin{equation}\label{e:tildetau}
\tilde\tau(k)A=\tau(k)A\tau(k\inv) \qquad \forall k\in K
\end{equation}
and decompose $\mathrm{End}(\spaziorp_\tau)$ into  irreducible invariant subspaces. The lower dimensional nontrivial subspace contains 
an element which turns out to be important in our analysis, because it coincides with the Fourier transform
of one of the generators $D^\tau_q$ at the basepoint. We shall  denote this element by $\Btau$.
The main result is Lemma~\ref{diagonalWnj} where we determine which polynomials in $\Btau$ 
lie in the various irreducible invariant subspaces, and Proposition~\ref{spettro}, where we determine the Gelfand spectrum $\Sigma^\tau_\cD$.

\subsection{The case where $n=3$}\label{sub:3}
We denote by $\tau_\emme$, with $\emme\in(1/2)\bN$ the irreducible representation of $SU_2$ of dimension $2\emme+1$. For $\emme\in\bN$, $\tau_\emme$ projects to a representation of $SO_3\equiv SU_2/\bZ_2$ and $\widehat{SO_3}=\{\tau_\emme:\emme\in\bN\}$.

The restriction of $\tau_\emme$ to $K_\base\equiv SO_2$ in \eqref{restriction} decomposes as
$$
{\tau_\emme}_{|_{SO_2}}=\sum_{s=0}^{2\emme}\sigma_s\ ,
$$
where $\sigma_s\left(\begin{bmatrix}\cos t&-\sin t\\\sin t&\cos t\end{bmatrix}\right)$ is the character $e^{i(s-\mu)t}$.
So $\spaziorp_{\tau_\mu}=\spaziorpU_{\sigma_0}\oplus\cdots\oplus \spaziorpU_{\sigma_{2\emme}}$
and the spaces $\spaziorpU_{\sigma_s}\sim \bC$ are all one dimensional.

We choose the basis of $\fs\fo_3$  formed by the vectors
\[
X_1=\begin{bmatrix}0&0&0\\0&0&-1\\0&1&0\end{bmatrix}
\qquad
X_2=\begin{bmatrix}0&0&1\\0&0&0\\-1&0&0\end{bmatrix}
\qquad
X_3=\begin{bmatrix}0&-1&0\\1&0&0\\0&0&0\end{bmatrix}
\]
and we identify
\[
X=\begin{bmatrix}0&-x_3&x_2\\x_3&0&-x_1\\-x_2&x_1&0\end{bmatrix}=\sum_{j=1}^3 x_j \, X_j\in\fs\fo_3\ 
\]
with $x=(x_1,x_2,x_3)\in\bR^3$. The following family of 
 invariant polynomials  of $\cP(\fg)^{{\rm Ad}(K)}$ is well known
$$
p_1=|x|^2\ ,\qquad q_1=|y|^2\ ,\qquad q_2=x\cdot y\ .
$$

Then the symmetrization map of Section~\ref{sect:invdiffop} produces  
$D_{p_1}=-\sum_{j=1}^3X_j^2$, a constant multiple of the Casimir operator on $SO_3$, and 
\[
\xi'_{\tau_\emme}=\emme(\emme+1).
\] 
According to Lemma~\ref{lem:simmetrizzaz-vettoriale}, we obtain  the following generating system of $\big(\bD(\bR^3)\otimes {\rm End}(\spaziorp_\tau)\big)^{K}$
$$
D_{q_1}^{\tau_\emme}=\Delta I\ ,\qquad \DD_{\tau_\emme}=D_{q_2}^{\tau_\emme}=-\sum_{j=1}^3\de_{y_j}\,d\tau_\emme(X_j)\ .
$$

Identifying ${\rm End}(\spaziorp_{\tau_\emme})$ with  $ \spaziorp'_{\tau_\emme}\otimes \spaziorp_{\tau_\emme}$, 
the representation
 $\tilde\tau_\emme$  defined in~\eqref{e:tildetau}
 is equivalent to the tensor product $\tau'_{\emme}\otimes\tau_{\emme}$, where $\tau'_{\emme}\sim\tau_{\emme}$ is the  contragredient representation of $\tau_{\emme}$.

Therefore $\tilde\tau_\emme$  decomposes as
$$
\tilde\tau_\emme \sim\tau'_{\emme}\otimes\tau_{\emme}\,\sim\,\tau_{2\emme}\oplus \tau_{2\emme-1}\oplus\cdots\oplus\tau_1\oplus\tau_0\ ,
$$
and accordingly 
\be\label{Pnn}
{\rm End}(\spaziorp_{\tau_\emme}) 
=\sum_{\ell=0}^{2\emme}\spaziW^{\ell}_{\tau_\emme}\ ,\qquad \spaziW^{\ell}_{\tau_\emme}\sim\tau_{\ell},
\ee
where $\mathrm{dim}\,\spaziW^\ell_{\tau_\emme}=2\ell+1$.

\medskip
Let 
\begin{equation}\label{Btauemme}
B_{\tau_\emme}=\widehat{\DD_{\tau_\emme}}(\base)=-id\tau_\emme(\baseuno )=\sum_{s=0}^{2\emme}(-\emme+s)P_{s,\tau}\ ,
\end{equation}
and consider the $\tilde\tau_{\emme}$-invariant subspace $\spaziW$ generated by $d\tau_\emme(\baseuno )$. Then
$$
\spaziW=\spaz_\bC\big\{\tau_\emme(k)d\tau_\emme(\baseuno )\tau_\emme(k)\inv\, :\,k\in \grpSO_3\big\}=d\tau_\emme(\mathfrak{so}_3) \ ,
$$
which is a 3-dimensional invariant subspace. So it must coincide with the component $\spaziW_{\tau_\emme}^1$ in~\eqref{Pnn}.

For every $\ell=0,\ldots ,2\emme$, let $\diagonali^\ell_{\tau_\emme}$ be the subspace of $K_\base$-invariant 
elements in $\spaziW^\ell_{\tau_\emme}$. 
Since the representation $\tilde\tau_{\emme}$ restricted to $\spaziW^\ell_{\tau_\emme}$ is equivalent to $\tau_{\ell}$, it contains the null weight with multiplicity one, and so $\diagonali^\ell_{\tau_\emme}$ is one-dimensional.  In particular, $B_{\tau_\emme}$ spans $\diagonali^1_{\tau_\emme}$.

\medskip
\subsection{The case where $n=4$}
When $n=4$, it is convenient to use the identifications
$$
SO_4\cong (SU_2\times SU_2)/\bZ_2\cong ({\rm Sp}_1\times {\rm Sp}_1)/\bZ_2\ ,\qquad \bR^4\cong\bH\ ,
$$
with one copy of ${\rm Sp}_1=\{u\in\bH:|u|=1\}$ acting by left multiplication and the other by right multiplication on $\bH$,
\be\label{action}
(u,v)\cdot y=uyv\inv\ .
\ee

According to the previous identifications, the unitary dual $\widehat{SO_4}$ consists of the exterior tensor products $\tau_{\enne,\emme}=\tau_\enne\boxtimes\tau_\emme$, with $\enne,\emme\in(1/2)\bN$ and $\enne+\emme\in\bN$.

The Lie algebra $\mathfrak{so}_4$ can be identified with $\fs\fu_2\times\fs\fu_2\cong \IM \bH\times\IM\bH$. For 
$\alpha, \beta\in \IM\bH$ we shall write
$X_{\alpha, \beta}$ for the corresponding element in $\mathfrak{so}_4$.
Differentiating \eqref{action} we obtain
\[
X_{\alpha, \beta}y=\al y-y\beta\ .
\]
Let
$$
\begin{bmatrix}i&0\\0&-i\end{bmatrix}
\ ,\qquad \begin{bmatrix}0&1\\-1&0\end{bmatrix}
\ ,\qquad \begin{bmatrix}0&i\\i&0\end{bmatrix} \ ,
$$
 be the standard basis of $\fs\fu_2$. Denote these elements by $U_1,U_2,U_3$ if they belong to the first copy of $\fs\fu_2$, and by $V_1,V_2,V_3$ if they belong to the second. 
 The adjoint action of $SO_4$ on $\mathfrak{so}_4\times\bR^4$
is given by
$$
\big({\rm Ad}(u,v)\big)(\al,\beta,y)=(u\al u\inv,v\beta v\inv,uyv\inv)\ ,
$$
and we choose the following family of generating  polynomials of $\cP(\fg)^{{\rm Ad}(K)}$ 
$$
p_1=|\al|^2\ ,\qquad p_2=|\beta|^2\ ,\qquad q_1=|y|^2\ ,\qquad q_2=\RE(\al y\,\overline{y\beta})=\lan \al y,y\beta\ran\ .
$$
Then $\la(p_i)=C_i$, a constant multiple of Casimir operator on the $i$-th copy of $SU_2$, so
$$
\xi'_{\tau_{\enne,\emme}}=\big(2\enne(2\enne+2),2\emme(2\emme+2)\big)\ 
$$
and the ${\rm End}(\spaziorp_{\tau_{\enne,\emme}})$-valued differential operators are
\[
D^{\tau_{\enne,\emme}}_{q_1}=\Delta I\qquad\qquad
\DD_{\tau_{\enne,\emme}}=D^{\tau_{\enne,\emme}}_{q_2}\ .
\]
One could work out an explicit formula for $\DD_{\tau_{\enne,\emme}}$, but for our purposes
it will suffice to know that for $\base=1\in\bH$
\[
\widehat{\DD_{\tau_{\enne,\emme}}}(\base)=-\sum_{j=1}^3d\tau_\enne(U_j)\otimes d\tau_\emme(V_j)\ ,
\]
which follows immediately from Lemma~\ref{lem:simmetrizzaz-vettoriale}.

The stabilizer $K_\base$ of $\base$
 is ${\rm diag}(SU_2\times SU_2)/\bZ_2\cong SO_3$. Since $(\tau_\enne\boxtimes\tau_\emme)_{|_{{\rm diag}(SU_2\times SU_2)}}=\tau_\enne\otimes\tau_\emme$, the analogue of formula~\eqref{restriction} becomes
$$
{\tau_{\enne,\emme}}_{|_{K_\base}}=\tau_{\emme+\enne}\oplus \tau_{\emme+\enne-1}\oplus\cdots\oplus  \tau_{|\emme-\enne|}\ .
$$
Therefore for $j=0,\ldots,2(\emme\wedge\enne)(=2\min\{\emme,\enne\})$,
we let $\sigma_j\sim \tau_{\emme+\enne-j}$ and $\spaziorpU_{\sigma_j}$ for its representation space.

Since every $\grpSU$-factor acts on the corresponding factor space, for the action on ${\rm End}(\spaziorp_{\tau_{\enne,\emme}})$ we have the decomposition
$$
{\tilde{\tau}}_{\enne,\emme}\sim\tau_{\enne,\emme}'\otimes\tau_{\enne,\emme}\sim \sum_{j=0}^{2\enne}\sum_{k=0}^{2\emme} \tau_{j}\boxtimes\tau_{k}
$$
corresponding to
$$
{\rm End}(\spaziorp_{\tau_{\enne,\emme}})=\sum_{j=0}^{2\enne}\sum_{k=0}^{2\emme} \cW^{j,k}_{\tau_{\enne,\emme}}\ ,\qquad \cW^{j,k}_{\tau_{\enne,\emme}}=\cW^j_{\tau_\enne}\otimes\cW^k_{\tau_\emme}\sim \tau_{j}\boxtimes\tau_{k}\ .
$$

We remark that, restricting the representation to $K_\base$, $\cW^{j,k}_{\tau_{\enne,\emme}}$ 
contains a fixed element if and only if $j=k$.

For every $\ell=0,1,\ldots ,2(\emme\wedge \enne)$, we denote by $\diagonali^\ell_{\tau_{\enne,\emme}}$  the subspace of $K_\base$-fixed matrices in $\spaziW^{\ell,\ell}_{\tau_{\enne,\emme}}$. The subspace $\diagonali^\ell_{\tau_{\enne,\emme}}$ is one-dimensional and consists of operators which are constant multiples of the identity operator on each $\spaziorpU_{\sigma_s}$.

Let
\[
B_{\tau_{\enne,\emme}}=\widehat{\DD_{\tau_{\enne,\emme}}}(\base)=-\sum_{j=1}^3d\tau_\enne(U_j)\otimes d\tau_\emme(V_j)\ .
\]
As noted in the previous subsection, $d\tau_\enne(U_j)$ is in $\spaziW_{\tau_\enne}^1$. Therefore 
$B_{\tau_{\enne,\emme}}$ is the only  nontrivial  
element in $\diagonali^1_{\tau_{\enne,\emme}}$.

Since $K_\base$--invariant elements
 act as constants on $\spaziorpU_{\sigma_s}$,
we may  define  coefficients $\auttau{s,{\tau_{\enne,\emme}}}$ so that
$$
B_{\tau_{\enne,\emme}}=\sum_{s=0}^{2(\emme\wedge\enne)}\auttau{s,\tau_{\enne,\emme}}P_{s,\tau_{\enne,\emme}}\ .
$$

\begin{lemma}\label{autovSO4}
 $\auttau{s,\tau_{\enne,\emme}}=4(\emme-s)(\enne-s)-2s(s+1)$.
\end{lemma}

\begin{proof}
In the polynomial model  $\spaziorp_{\tau_\enne}\otimes \spaziorp_{\tau_\emme}=\cP^{2\enne,2\emme}(\bC^2)$, we obtain
$$
\spaziorpU_{\sigma_s}=|z|^{2s}\,\,\cH^{2\enne-s,2\emme-s}\ 
$$
where $\cH^{\alpha,\beta}$ denotes the space of harmonic polynomials of bidegree $(\alpha,\beta)$.

Then $d\tau_\enne(U_j)$ and $d\tau_\emme(V_j)$ act respectively  on the holomorphic and the antiholomorphic variables.
Therefore we shall more briefly write
$Z_j=d\tau_\enne(U_j)$ and  $\bar Z_j=d\tau_\emme(V_j)$ so that
\beas
&Z_1=-iz_1\de_{z_1}+iz_2\de_{z_2}\ ,\qquad &Z_2=-z_2\de_{z_1}+z_1\de_{z_2}\ ,\qquad &Z_3=-iz_2\de_{z_1}-iz_1\de_{z_2}\\
&\bar Z_1=i\bar z_1\de_{\bar z_1}-i\bar z_2\de_{\bar z_2}\ ,\qquad &\bar Z_2=-\bar z_2\de_{\bar z_1}+\bar z_1\de_{\bar z_2}\ ,\qquad &\bar Z_3=i\bar z_2\de_{\bar z_1}+i\bar z_1\de_{\bar z_2}\ ,
\eeas
and
\beas
&Z_2+iZ_3=2z_1\de_{z_2}\ ,\qquad &Z_2-iZ_3=-2z_2\de_{z_1}\\
&\bar Z_2+i\bar Z_3=-2\bar z_2\de_{\bar z_1}\ ,\qquad &\bar Z_2-i\bar Z_3=2\bar z_1\de_{\bar z_2}\ .
\eeas

According to the above notation
\beas
B_{\tau_{\enne,\emme}}&=-\sum_{j=1}^3Z_j\bar Z_j=-Z_1\bar Z_1-\half(Z_2+iZ_3)(\bar Z_2-i\bar Z_3)-\half(Z_2-iZ_3)(\bar Z_2+i\bar Z_3)\\
&=-(z_1\de_{z_1}-z_2\de_{z_2})(\bar z_1\de_{\bar z_1}-\bar z_2\de_{\bar z_2})-2z_1\bar z_1\de_{z_2}\de_{\bar z_2}-2z_2\bar z_2\de_{z_1}\de_{\bar z_1}\ .
\eeas

Consider the polynomial in $\spaziorpU_{\sigma_s}$ given by 
$$
P(z,\bar z)=|z|^{2s}z_1^{2\enne-s}\bar z_2^{2\emme-s}=\sum_{k=0}^s\binom{s}{k}z_1^{2\enne-s+k}\bar z_1^k z_2^{s-k}\bar z_2^{2\emme-k}\ .
$$
Explicit computations show that
\[
B_{\tau_{\enne,\emme}}P=\big(4(\emme-s)(\enne-s)-2s(s+1)\big)P\ .\qedhere
\]
\end{proof}

\subsection{Polynomials in $\Btau$} 
In order to treat  the cases $n=3,4$  together as far as possible, it is convenient to use the following notation.
We recall that
\[
\finetau=\begin{cases}2\emme & n=3\quad \tau=\tau_{\emme\phantom{,\enne}}
\\ 2\,(\emme\wedge \enne) &n=4 \quad \tau=\tau_{\enne,\emme}\ 
\end{cases}
\]
and for $j=0,\ldots ,\finetau$ we put
\begin{equation*}\label{Pnn2}
\cW^j_\tau=
\begin{cases}
\cW^{j}_{\tau_{\emme}} &n=3\quad \tau=\tau_{\emme\phantom{,\enne}}
\\
\cW^{j,j}_{{\tau_{\enne,\emme}}} &n=4 \quad \tau=\tau_{\enne,\emme}\ .
\end{cases}
\end{equation*}

\begin{lemma}\label{diagonalWnj}
For every $\ell=0,\ldots ,a_\tau$, there is a monic polynomial $q_{\tau}^\ell$ of degree $\ell$ such that $\diagonali^\ell_{\tau}$ consists of the scalar multiples of  $q_{\tau}^\ell(\Btau)$.

Moreover, for a polynomial $p$ of degree at most $\finetau$, we have $p(\Btau)\in \sum_{j\le \ell}\spaziW_{\tau}^j$ if and only if $\deg(p) \le \ell$.\\
Finally, when $n=3$, $q_{\tau}^\ell$ is of the same parity as $\ell$.
\end{lemma}

\begin{proof} We start giving the proof in the case where $n=3$, so $\tau=\tau_\mu$ and $\Btau=B_{\tau_\emme}$, given in~\eqref{Btauemme}.

We proceed by induction
proving all statements together. They are trivial for $\ell=0$ and have been proved in subsection~\ref{sub:3} for $\ell=1$.

Let
 $\ell\geq 2$ and consider the polynomial $t\mapsto p(t)=t\, q_\tau^{\ell-1}(t)$. Then $p$ 
  is of degree $\ell$ and $p(\Btau)$ is a $K_\base$--invariant element.
By the inductive hypothesis, $q_\tau^{\ell-1}(\Btau)$ is in $\spaziW_{\tau_\emme}^{\ell-1}$
and $\Btau$ is in $\spaziW_{\tau_\emme}^{1}$; moreover we have  $\tau_\emme'\otimes\tau_\emme|_{\spaziW_{\tau_\emme}^{\ell-1}}\sim \tau_{\ell-1}$ and the decomposition
\[
\tau_{\ell-1}\otimes \tau_1\sim \tau_{\ell}\oplus \tau_{\ell-1} \oplus \tau_{\ell-2}
\]
in irreducible summands with multiplicity one for  $\tau_{\ell}$.
Therefore $p(\Btau)$ is  in 
$\spaziW_{\tau_\emme}^\ell\oplus\cdots\oplus \spaziW_{\tau_\emme}^0$ with a nontrivial component in $\spaziW_{\tau_\emme}^\ell$.
We call $q_{\tau}^\ell(\Btau)$ this component. Hence $(p-q_{\tau}^\ell)(\Btau)$ is in $\sum_{j\in L} \spaziW_{\tau_\emme}^j$, which, by the inductive hypothesis, is a polynomial in $\Btau$ of degree at most $\ell-1$. Then $q_{\tau}^\ell$ has degree $\ell$ and its leading term is the same as $p$, so it is monic. 

As for the parity statement,
 note that the inner product is proportional to the  Hilbert--Schmidt one so that
\[
\langle (\Btau)^\ell,(\Btau)^{\ell'}\rangle=c\, \sum_{j=-\emme}^\emme j^{\ell+\ell'}=0
\]
when $\ell+\ell'$ is odd.

In the case where $n=4$, the proof is almost identical to the previous one; the only difference lies in  
the fact that in the decomposition
\[
\begin{aligned}
(\tau_{\ell-1}\boxtimes\tau_{\ell-1})\otimes (\tau_{1}\boxtimes\tau_{1})
&\sim (\tau_{\ell}\oplus \tau_{\ell-1} \oplus \tau_{\ell-2})\boxtimes( \tau_{\ell}\oplus \tau_{\ell-1} \oplus \tau_{\ell-2})
\\
&\sim \sum_{j=\ell-2}^{\ell}\sum_{k=\ell-2}^{\ell}\tau_{j}\boxtimes\tau_{k}
\end{aligned}
\]
in irreducible summands with multiplicity one for  $\tau_{\ell}\boxtimes\tau_{\ell}$, $K_\base$--invariant elements
arise only when $j=k$.
\end{proof}

\subsection{Equivariant homogeneous polynomials}
\quad
Denote by $Q_\tau$ the $K$--equivariant polynomial arising from the symmetrization of the polynomial $q_2$ and Lemma~\ref{lem:simmetrizzaz-vettoriale}. Explicity, 
\be\label{DtauQ}
\widehat{\DD_\tau F}=\widehat{F}\,Q_\tau \qquad \forall F\in \Smatr.
\ee

This polynomial is also completely determined by  the following properties:
\begin{itemize}
\item[i)]  $Q_\tau (\base)=\Btau$;
\item[ii)] $Q_\tau$ is $K$-equivariant;
\item[iii)] $Q_\tau$ is of least degree.
\end{itemize}
Note that $Q_\tau$ is homogeneous of degree $\undue$ where
$$
\undue=
\begin{cases}
1 \quad \text{ if } n=3
\\
2\quad  \text{ if } n=4
\end{cases}.
$$

Suppose now that $P:\bR^n\to {\rm End}(\spaziorp_\tau)$ is a $K$-equivariant polynomial. 
Then each homogeneous component of $P$ 
 is also equivariant and $P$ can be uniquely decomposed 
as the sum of $\spaziW_{\tau}^\ell$-valued homogeneous polynomials.

\begin{proposition}\label{p:gradopolinomi} Suppose that there exists  a nontrivial $K$--equivariant polynomial,
homogeneous of degree $d$ and $\spaziW_{\tau}^\ell$--valued. Then $d\geq \ell\undue$
and $d-\ell\undue$ is even.
\end{proposition}

\begin{proof} Suppose that $P:\bR^n\to {\rm End}(\spaziorp_\tau)$ is a nontrivial $K$--equivariant polynomial
homogeneous of degree $d$ and $\spaziW_{\tau}^\ell$-valued.
Since  $P(\base)$ is in $\diagonali_\tau^\ell$, by Lemma~\ref{diagonalWnj}
\[
P(\base)= c\,\, q_\tau^\ell(\Btau)
\]
for some constant $c$. Setting $y'=y/|y|$ for $y\ne 0$, 
$$
P(y)=c\, |y|^{d}\tau(k_{y'})q_\tau^\ell(\Btau)\tau(k_{y'}\inv)=c\, |y|^{d} q_\tau^\ell(Q_{\tau}(y/|y|)).
$$
Since in the last equation
there are no negative powers, we deduce that $d\geq \ell\undue$.
Finally, since $P$ is a polynomial,
$d-\ell\undue$ is even. 
\end{proof}

\bigskip

\subsection{${\rm End}(\spaziorp_\tau)$-valued spherical functions as derivatives of scalar spherical functions}
\quad
\medskip 
  
For $\ro\ge0$ we denote by $\varphi_\ro$ the  spherical function of $(G,K)$ as  an ordinary Gelfand pair with eigenvalue $\ro^2$ relative to $\Delta$.
 In formulae,
$$
\varphi_\ro(y)=\varphi_1(\ro y)\quad\text{ where }\quad \varphi_1(y)=\int_{S^{n-1}}e^{i\, y\cdot \eta}\, d\sigma (\eta)\ \ \  \Big(=|y|^{-\frac{n}{2}+1}\,J_{\frac{n}{2}-1}\big(|y|\big)\Big)
$$
and $\sigma$ is the normalized surface measure of the unit sphere $S^{n-1}$ in $\bR^n$.

Coherently with Lemma~\ref{autovSO4}, we call $\aut{s}$ the eigenvalue of $\Btau$ on the subspace $\spaziorpU_{\sigma_s}$.
In formulae,
\[
\aut{s}=
\begin{cases}
-\emme+s & n=3,\quad \ \tau=\tau_{\emme\phantom{,\enne}}
\\
4(\emme-s)(\enne-s)-2s(s+1) & n=4,\quad \ \tau=\tau_{\enne,\emme}
\end{cases}
\qquad s=0,\ldots,\finetau,
\]
so that $B_\tau=\sum_{s=0}^{\finetau} \aut{s}P_{s,\tau}$. Note that, with respect to the notation in~\eqref{diagonal} $\aut{s}$ is a shorthand notation
\begin{equation}\label{e:shorthand}
\aut{s}=b_{s,\tau}(Q_\tau,1).
\end{equation}
Since $\aut{s}\neq \aut{s'}$ when $s\neq s'$, 
the orthogonal projection $P_{s,\tau}$ of $\spaziorp_\tau$ onto $\spaziorpU_{\sigma_s}$ can be written as a polynomial
in $\Btau$ of degree $\finetau$. Indeed,
\[
P_{s,\tau}=p_{s,\tau}(\Btau)\qquad \text{ where }\qquad p_{s,\tau}(t)=\prod_{j\neq s,\, j=0}^{\finetau} \frac{t-\aut{j}}{\aut{s}-\aut{j}}\ .
\]

In the following proposition we write the ${\rm End}(\spaziorp_\tau)$-valued spherical functions as derivatives of scalar spherical functions and we observe that 
the distributional Fourier transform of   $\Phi^\tau_{\ro,s}$  is given in terms of the orthogonal projection $P_{s,\tau}$. As a consequence we easily compute   the embedded spectrum $\Sigma^\tau_\cD$ showing that it consists of 
half parabolas when $n=3$ and rays through the origin when $n=4$.
\smallskip

Given an $\textup{End}(\spaziorp_\tau)$-valued polynomial $P$, we denote by   $P(\partial)$  the differential operator defined by the rule
\[
\widehat{P(\partial)F}(\eta)=\widehat F(\eta)\,P(\eta)\qquad \eta\in \bR^n, \quad \forall F\in \Smatr.
\]

In particular note that
$$
\DD_\tau=Q_\tau(\partial).
$$

 We recall that $\undue=1$ when $n=3$,
and $\undue=2$ when $n=4$. 
\begin{proposition}\label{spettro}
The spectrum $\Sigma^\tau_\cD$ is given by
\[
\Sigma^\tau_\cD=\big\{\big(\ro^2,\aut{s}\,\ro^\undue\big)\,:\,\ro\ge0\,,\ s=0,\dots,\finetau \big\}.
\]
Moreover, the ${\rm End}(\spaziorp_\tau)$-valued spherical function $\Phi^\tau_{\ro,s}$ in~\eqref{Phi} corresponds to the pair of eigenvalues $\big(\ro^2,\aut{s}\,\ro^\undue\big)$ and satisfies
\begin{equation}\label{Phixi}
\Phi^\tau_{\ro,s}=(2\pi)^{n}\,\frac{d_\tau}{d_{\sigma_s}}\,
(p_{s,\tau}\circ Q_{\tau})(\partial)\, \varphi_\ro
\qquad \forall \ro>0,   
\end{equation}
\end{proposition}

\begin{proof} The case $\ro=0$ is trivial. When $\ro>0$, by  formula~\eqref{Phi} we have

\als{
\Phi^\tau_{\ro,s}(y)&=\frac{d_\tau}{d_{\sigma_s}} \int_{\grpSO_n} e^{i\ro\,\base\cdot(k y)}\tau (k\inv)P_{s,\tau}\,\tau (k)\,dk\ 
\\
&=\frac{d_\tau}{d_{\sigma_s}} \int_{\grpSO_n} e^{i\ro\, k\inv \base\,\cdot y}\,(p_{s,\tau}\circ Q_{\tau})(k\inv\base)\,dk
\\
&=\frac{d_\tau}{d_{\sigma_s}} \int_{S^{n-1}} e^{i\ro\,\eta\,\cdot y}\,(p_{s,\tau}\circ Q_{\tau})(\eta)\, d\sigma(\eta)
\\
&=
(2\pi)^{n}\,\frac{d_\tau}{d_{\sigma_s}}\,
(p_{s,\tau}\circ Q_{\tau})(\partial)\, \varphi_\ro\ (y),
} 
We now determine the embedded spectrum.
From formula~\eqref{Phi}, 
it is easy to check that $\Delta \Phi^\tau_{\ro,s}=\ro^2\, \Phi^\tau_{\ro,s}$.
As for the second eigenvalue by \eqref{Phixi}
\begin{equation*} 
\widehat \Phi^\tau_{1,s}(\base)=C_{s,\tau}\, P_{s,\tau}
=C_{s,\tau}\, \, p_{s,\tau}(\Btau). 
\end{equation*}
where $C_{s,\tau}=\frac{d_\tau}{d_{\sigma_s}}\,\frac{(2\pi)^n}{|S^{n-1}|}$. Hence by 
\eqref{diagonal} and
 \eqref{e:shorthand} 
we have
$$
\widehat{\DD_\tau\Phi^\tau_{1,s}}(\base)=\widehat\Phi^\tau_{1,s}(\base)\, Q_{\tau}(\base)
= C_{s,\tau}
\,
 P_{s,\tau}\, B_{\tau}
=\aut{s}\,\widehat\Phi^\tau_{1,s}(\base).
$$
By  homogeneity and equivariance we have
\[
\DD_\tau\Phi^\tau_{\ro,s}=\ro^\undue \aut{s}\,\Phi^\tau_{\ro,s}\qquad \forall \ro>0, \quad s=0,\ldots \finetau .
\qedhere
\]
\end{proof}
\bigskip

Since we regard the $\tau$-spherical transform as defined on the embedded spectrum,  and according to Lemma~\ref{Phi}, we shall write 
\begin{equation}\label{cGtau}
\cG_\tau^\sharp F(\ro^2,\aut{s}\ro^\undue)
=\frac1{d_\tau}\,\int_{\bR^n}\tr\big(F(y)\Phi^\tau_{\ro ,s}(-y)\big)\,dy  = b_{s,\tau}\big(\widehat F,\ro\big)\ .
\end{equation}
}

The next corollary shows that we can  disregard part of the spectrum when dealing with even or odd functions in the case where $n=3$.

\begin{corollary}\label{Phi3}
In the case where $n=3$,
\[
\Phi^{\tau}_{\ro,\finetau-s}(y)=\Phi^{\tau}_{\ro,s}(-y),
\]
so that if $F$ is decomposed into its even part $F_+$ and its odd part $F_-$, then
\[
\cG_\tau^\sharp F_\pm(\ro^2, \aut{s}\ro)
=\pm \cG_\tau^\sharp F_\pm(\ro^2,- \aut{s}\ro)
\]
\end{corollary}

\begin{proof}
 Note that when $n=3$
 \[
p_{\finetau-s}(t)=p_s(-t)
\]
and since $Q_{\tau}$ is homogeneous of degree $\undue=1$,
\[
\big(p_{\finetau-s}\circ Q_{\tau}\big)(y)=\big(p_s\circ Q_{\tau}\big)(-y)\ .
\]
Moreover $d_{\sigma_s}=1$, so that
\als{
\Phi^\tau_{\ro,\finetau-s}(y)&
=d_\tau\, \int_{S^2}e^{i\ro\, y\cdot\eta }\,p_{\finetau-s}\circ Q_{\tau}(\eta)\,d\sigma(\eta)
\\
&=d_\tau\, \int_{S^2}e^{i\ro\, y\cdot\eta}\,p_{s}\circ Q_{\tau}(-\eta)\,d\sigma(\eta)
\\
&=d_\tau\, \int_{S^2}e^{-i\ro\, y\cdot\eta}\,p_{s}\circ Q_{\tau}(\eta)\,d\sigma(\eta)
\\
&=\Phi^\tau_{\ro,s}(-y).
} 
Finally, by a simple change of variables
\als{
\cG_\tau^\sharp F_\pm(\ro^2, \aut{s}\ro)
&=\pm \frac1{d_\tau}\,\int_{\bR^n}\tr\big(F_\pm(y)\Phi^{\tau}_{\ro,\finetau-s}(-y)\big)\,dy,
}
and the desired formula follows remembering that, when $\tau=\tau_\emme$,  $\finetau=2\emme$ and $\aut{s}=-\emme+s$, so 
$\aut{\finetau-s}=-\aut{s}$.
\end{proof}

\bigskip
\section{Schwartz correspondence for $(G,K,\tau)$}
\medskip

In this section   we give our first argument to prove that  $ \cG_\tau^\sharp $ maps $\Smatr $ into $\cS(\Sigma^\tau_\cD)$.  The implied norm inequalities will be left implicit because their dependence on $\tau$ is not sufficiently sharp for our purposes. Sharper estimates
 will be obtained in Section~\ref{finale}.

The  analysis of the Banach algebra $L^1(\bR^n,{\rm End}(\spaziorp_\tau))^K$ developed in the previous section
leads to the decomposition of ${\rm End}(\spaziorp_\tau)$--valued equivariant functions into a sum of derivatives of scalar valued functions.
\begin{proposition}\label{sommaF2}
Let $F$ be 
in $\cS(\bR^n,{\rm End}(\spaziorp_\tau))^K$.
Then $F$ can be expressed in a unique way as
\[
 F=\sum_{i=0}^{\finetau}\DD_\tau^ig_i\ ,
\]
with scalar valued functions $g_i\in \cS(\bR^n)^K$.
\end{proposition}

\begin{proof} Every  $F$ in $\Smatr$ can be decomposed  in a unique way
 as $F=\sum_{\ell=0}^{\finetau} F_\ell$, where the component $F_\ell$ is in  $\SmatrW$.
 Therefore we start with a single  $F_\ell$.

Since the Fourier transform commutes with rotations, $\widehat{F_\ell}$ is 
in $\SmatrW$ too, so that, by Lemma~\ref{diagonalWnj}, for every $\ro\neq 0$, $\widehat{F_\ell}(\ro\base)$ is a constant multiple, $f_\ell(\ro)$ say,
of $q_\tau^\ell(\Btau)$.

Suppose that $P_{\ell,d}$ is the Taylor polynomial of $\widehat{F_\ell}$ centred at the origin of order $d$.
 i.e.,
\[
\widehat{F_\ell}(\eta)=P_{\ell,d}(\eta)+o(|\eta|^d) \qquad \eta\to 0.
\]
Then $P_{\ell,d}$ and all of its homogeneous components are $K$--equivariant and $\spaziW_\tau^\ell$--valued.
Therefore 
$
P_{\ell,d}(\base)
$
is a constant multiple of $q_\tau^\ell(\Btau)$. Using Proposition~\ref{p:gradopolinomi} we conclude that $P_{\ell,d}=0$ if $d<\undue\ell$.

We can therefore write 
\[
\widehat{F_\ell}(\ro\base)=f_\ell(\ro)\,q_\tau^\ell(\Btau)=c_\ell\, \ro^{\undue\ell}\, q_\tau^\ell(\Btau)+o(\ro^{\undue\ell}) \qquad \ro\to 0\ .
\]
Clearly the so-obtained function $f_\ell$ is in $\cS(\bR)$ and
of the same parity as $\undue\ell$.

Hence
\[
f_\ell(\ro)=c_\ell\, \ro^{\undue \ell}+o(\ro^{\undue\ell}) \qquad \ro\to 0\ ,
\] 
and by Hadamard's division Lemma~\cite{Jet}
there exists an even smooth function $\tilde f_\ell$ on $\bR$ such that $f_\ell(\ro)=\ro^{\undue\ell}\, \tilde f_\ell(\ro)$ and $\tilde f_\ell(0)=c_\ell$.
With $h_\ell=\tilde f_\ell(|\cdot|)$, we conclude that when $\eta=|\eta|\eta'$
\[
\widehat{F_\ell}(\eta)= \tilde f_\ell (|\eta|)\, |\eta|^{\undue\ell}\,\tau(k_{\eta'})\,q_\tau^\ell(\Btau)\tau(k_{\eta'})^*
= h_\ell(\eta)\,q^\ell_\tau(Q_\tau(\eta)).
\]
Therefore $h_\ell$ is a uniquely determined scalar invariant Schwartz function on $\bR^n$ and 
taking the inverse Fourier transform, for a scalar invariant Schwartz function $\gamma_\ell$ on $\bR^n$,
\[
F_\ell(y)=q^\ell_{\tau} (\DD_\tau)\, \gamma_\ell(y).
\]
Since $q^\ell_{\tau}$ is a polynomial of degree $\ell$, rearranging the terms in the sum 
\[
F=\sum_{\ell=0}^{\finetau} F_\ell=\sum_{\ell=0}^{\finetau} q^\ell_{\tau} (\DD_\tau)\, \gamma_\ell(y)
\]
we obtain the desired decomposition.
\end{proof}
 
\begin{lemma}\label{FourierF}
Let $F$ be in 
$\Smatr$. Then 
there  exist  $\gamma_0,\ldots, \gamma_{a_\tau}$ in $  \cS(\bR)$ such that 
$$
\cG_\tau^\sharp F(\ro^2, \aut{s}\ro^\undue)=
\sum_{i=0}^{\finetau}{\gamma_i}(\ro^2)\, 
(\ro^\undue\, \aut{s})^i
\qquad \forall \ro\geq 0, \quad s=0,1,\ldots,\finetau.
$$
\end{lemma}
 
 \begin{proof} 

Let $F$ be in 
$\Smatr$ and
$g_i$ be the $K$--invariant scalar functions we have associated to $F$ in Proposition~\ref{sommaF2}.
Then $\widehat{g_i}$ is $K$--invariant, so that by Schwartz--Mather there exists a  function $\gamma_i \in \cS(\bR)$ such that 
\[
\widehat{g_i}(\eta)=\gamma_i(|\eta|^2)
\]
and {by \eqref{DtauQ}}
\[
\widehat F (\ro\base)
=\sum_{i=0}^{\finetau}\widehat{g_i}(\ro\base)\, \widehat{\DD_\tau^i}(\ro\base)
= \sum_{i=0}^{\finetau}{\gamma_i}(\ro^2)\, \ro^{\undue i}\,B_\tau^i
= \sum_{i=0}^{\finetau}{\gamma_i}(\ro^2)\, \sum_{s=0}^{\finetau} (\ro^\undue\, \aut{s})^i\, P_{s,\tau}
\]
Finally, {by \eqref{cGtau},}
 $$
\cG_\tau^\sharp F(\ro^2, \aut{s}\ro^\undue)=
b_{s,\tau}(\widehat F,\ro)= \sum_{i=0}^{\finetau}{\gamma_i}(\ro^2)\, 
(\aut{s}\,\ro^\undue)^i
\qquad \forall \ro\geq 0
 $$
and the thesis follows.
\end{proof}
  
Given $f\in \cS(G)^{{\rm Int}(K)}$, via the map $A_\tau$, the following result 
gives the existence of
  an extension of $\cG_\tau f_\tau$  for each $\tau\in \widehat K$.
  
\begin{corollary}\label{estensioneGn} 
Let $F$ be in 
$\Smatr$ . Then
\begin{enumerate}[\rm(i)]
\item
 There exists $u\in\cS(\bR^2)$ such that 
${u}_{ |_{\Sigma_\cD^\tau}}=\cG^\sharp_\tau F$. 
 
\item The map $\cG^\sharp_\tau $ is an isomorphism from $\Smatr$  onto $\cS(\Sigma_\cD^\tau)$.
\item The map $\cG_\tau $ is an isomorphism from $\Stau$  onto $\cS(\Sigma_\cD^\tau)$.

\end{enumerate}
\end{corollary}
 
\begin{proof} Let $\cC_\tau$ be the closed convex hull of  $\Sigma_\cD^\tau$ and 
\[
\cC^1_\tau=\{(\xi_1,\xi_2)\, : \, \mathrm{dist}((\xi_1,\xi_2),\cC_\tau)<1\}\ .
\]
Note that if $(\xi_1,\xi_2)$ is in $\cC^1_\tau$ and  $\xi_1>1$, then $|\xi_2| \le R\, \xi_1$ for some $R>0$.

Let $\eta$ be a smooth function on $\bR^2$ which takes value 1 on
 $\cC_\tau $, vanishes outside  $\cC_\tau^1$ and with bounded derivatives  of any order.  
 Let $F$ be in $\Smatr$ and let $\gamma_0,\ldots, \gamma_n$ in $  \cS(\bR)$ be as in Lemma~\ref{FourierF} such that 
\[
\begin{aligned}
\cG_\tau^\sharp F\big(\ro^2, \aut{s}\,\ro^\undue\big)&=
\sum_{i=0}^{\finetau}(\aut{s}\,\ro^\undue)^i\,{\gamma_i}(\ro^2)\, 
\qquad \forall \ro\geq 0, \quad s=0,1,\ldots,a_\tau.
\end{aligned}
\]

Define the Schwartz  function   $g$   on $\bR^2$ by 
$$
u(\xi_1, \xi_2)=\eta(\xi_1, \xi_2) \, 
\sum_{i=0}^{\finetau}\xi_2^i\,{\gamma_i}\, (\xi_1)
\qquad \forall (\xi_1, \xi_2)\in \bR^2.
$$
Then  ${u}_{ |_{\Sigma_\cD^\tau}}=\cG^\sharp_\tau F$.

By Corollary~\ref{nucleo-tipo-tau}   the map $\cG_\tau^{-1}$ is continuous and (ii) and (iii) follow by Corollary~\ref{isoAtau}.
 \end{proof}

\bigskip
\section{Growth control for large  $\tau$ and conclusion}\label{finale}

In this section we conclude the proof of Theorem~\ref{main}.
As required in Theorem~\ref{rendiconti},
for each $\tau$ and each order $N$, given $f$ in $\cS(G)^{{\rm Int}(K)}$, we produce an extension of $\cG_\tau f_\tau$ with rapidly $\tau$-decaying norms up to order $N$. 
For this purpose we will follow steps (i) and (ii) in the Introduction.

\smallskip
We begin by observing that, as in the case of the euclidean Fourier transform, the derivatives of the spherical transform
$\cG_\tau f_\tau$  are controlled by Schwartz norms of $f_\tau$.
 This is shown in the next  lemma, whose use is twofold.
 On the one hand
 it gives   estimates of the coefficients of the formal power series $s_\tau$ in (i) of the introduction (Proposition \ref{taylor});
 on  the other hand it
  allows us   to  control uniformly in $\tau$ the norms  of the extension of $\cG_\tau f_\tau$ 
 under the assumption that $\cG_\tau f_\tau$ vanishes of infinite order at the origin and   (ii) will follow
 (Proposition~\ref{p:nulla})
 
  \begin{lemma}\label{Gderivatives} Let $\tau$ be in $ \widehat K$. 
There exist  sequences  $\{C_\der\}$ and $\{N_\der\}$, independent of $\tau$ and increasing,  such that for every $\der\in \bN$ 
$$
\left|(d/dt)^\der\cG_{\tau}f_\tau\big(t^{2/\undue},\aut{s}t\big)\right| 
\le C_{\der} \,\|f_\tau\|_{(N_\der)}\ ,
\qquad \forall t\geq 0\quad  \forall f_\tau\in \Stau.
$$

\end{lemma} 

\begin{proof} 

 Let $f_\tau\in \Stau$ and $F=A_\tau f_\tau$.
By~\eqref{cGtau}
$$
\cG_\tau f_\tau(t^{2/\undue},\aut{s}t)
=\cG_\tau^\sharp F(t^{2/\undue},\aut{s}t)=
   \frac{1}{d_{\sigma_s}}\,\tr\big(\widehat F(t^{1/\undue}\,\base) P_{s,\tau} \big).
$$
We recall that when $n=3$, then $\undue=1$ and the result is immediate. When $n=4$ we have $\undue=2$, but also
$\widehat F(-y)=\widehat F(y)$, for every $y\in \bR^4$. Therefore the thesis follows easily from Propostition~\ref{norme-f-Af}. 
\end{proof}

 \medskip 
We   introduce the following notation for the derivatives at the origin of a function $\psi \in \cS(\Sigma_\cD^\tau)$
 $$
c_{s,\tau, \der}(\psi)= \left(\frac{d}{dt}\right)^{\der}_{|_{t=0^+}}\psi\big(t^{2/\undue},\aut{s}t \big) 
 \qquad \der\in \bN,\quad s=0,1,\ldots a_\tau \ .
$$
 Note that the  previous lemma gives estimates of $c_{s,\tau, \der}(\cG_\tau f_\tau)$ for each $ f_\tau\in \Stau$.

\subsection{Control in $\tau$ of   derivatives at the origin}
In this section we prove the following statement.

 \begin{proposition}\label{taylor} 
Let $f_\tau$ be in $\Stau$ 
and $u $ a smooth extension of $\cG_\tau f_\tau$  to $\bR^2$.
Then
for every   integer $\der\geq 0 $ 
$$
\left| \binom{\der}{q} \partial_1^{\der-q}\partial_2^q u(0,0) \right| \le 
C_{\der}\, d_\tau^{\der}\,  \|f_\tau\|_{(N_{\der})}\ ,\qquad q=0,1,\ldots \der
$$
where   the  sequences $\{C_{\der}\}$ and $\{N_{\der}\}$ are independent of $\tau$ and    increasing.
\end{proposition}

The proof of Proposition~\ref{taylor} takes some effort. Indeed, its proof is rather technical and is based 
on relating (unknown) Maclaurin coefficients of any extension~$u$ up to order $\der$ with the (known) derivatives 
of $\cG_\tau f_\tau$ along the curves defining the spectrum $\Sigma^\tau_\cD\subset\bR^2$.
 
In doing so, for each order of derivation $\der$, one  obtains a (possibly overdetermined) linear system of the form
\[
\Lambda_{\tau,m} x= c_{\tau,m},
\]
where the matrix $\Lambda_{\tau,m}$ is expressed in terms of the integers $\aut{s}$ and the vectors    $x$ and $c_{\tau,m}$ are  related to
$ \partial_1^{\der-q}\partial_2^q u(0,0)$ and  $c_{s,\tau, \der}(\cG_\tau f_\tau)$, respectively.
 By Corollary~\ref{estensioneGn}, this system is solvable; our task here is to describe the norm of $x$ in terms of the norms of $c_{\tau,m}$.
 
 Estimates are split
in Section \ref{s:proof644} and \ref{s:proof643}, dealing  with the cases where $n=4$ and $n=3$, respectively.
In each case a special approach to the problem is used.
 
\subsubsection{Proof of Proposition~\ref{taylor} when $n=4$}\label{s:proof644}
 \quad
 In this case, the expression of the integers $\aut{s}$ requires some additional work, starting from the following technical lemma.

\begin{lemma}\label{traslazione} 
Let $n=4$. 
For any positive integer $ \ind\leq a_\tau$ there exists an integer $\mass_{\tau,\ind}$ such that 
if
\begin{equation}
\label{e:lambdaitau}
\la_{i,\tau,\ind}
=\aut{i}-\mass_{\tau,\ind}\qquad i=0,1,\ldots a_\tau,
\end{equation} 
then 
\[
\prod_{\substack{i=0\\ i\ne s}}^{\ind } 
\left|
\frac{\la_{i,\tau,\ind}}{\la_{i,\tau,\ind}-\la_{s,\tau,\ind}}\right|\leq 1
\qquad \forall s =0,\ldots  \ind\ .  
\]
Moreover $|\mass_{\tau,\ind}|\leq d_\tau$. 
\end{lemma}

\begin{proof} 
Let $\ind\leq a_\tau$ and   $s^*_{\ind}$   an index in $  \{0,1,\ldots \ind \}$   such that  
\[
\prod_{\substack{i=0\\ i\ne s^*_\ind}}^{\ind } | \aut{i}-\aut{s^\ast_\ind} | 
  \leq \prod_{\substack{i=0\\ i\ne {s}}}^{\ind } |\aut{i}- \aut{s}|\qquad\forall s=0,1,\dots \ind .
\]
Define $\mass_{\tau,\ind}:=\aut{s^\ast_\ind}$, then $|\mass_{\tau,\ind}|$ is an integer, no bigger than $d_\tau$.

Notice that  for $i=s^*_\ind$ we have
$\lambda_{i,\tau,\ind}=\aut{i}-\mass_{\tau,\ind}=\aut{i}-\aut{s^*_\ind} =0$,
and   $\lambda_{i,\tau,\ind}$ are nonzero  integers when $i\not=s^*_\ind$. Then, for every  $s$   in $ \{0,1,\ldots \ind\ \}$,  
\[
\prod_{\substack{i=0\\ i\ne s}}^{ \ind } 
\left|
\la_{i,\tau,\ind}\right|\leq \prod_{\substack{i=0\\ i\ne s^*_d}}^{ \ind } 
\left|
\la_{i,\tau,\ind}\right|=
\prod_{\substack{i=0\\i\ne s^*_\ind}}^\ind | \aut{i}-\aut{s^*_\ind} | 
\leq 
\prod_{\substack{i=0\\ i\ne s}}^{\ind } 
\left|
\aut{i}-\aut{s}\right|
=\prod_{\substack{i=0\\ i\ne s}}^{\ind } 
\left|
\la_{i,\tau,\ind}-\la_{s,\tau,\ind}\right| \ .
\qedhere
\]
\end{proof}

\bigskip
We now continue  the proof of Proposition~\ref{taylor}. 
 
 Fix $\tau\in\widehat K$ and let $f_\tau\in\Stau$ 
and $u $ a smooth extension of 
$\cG_\tau f_\tau$
 on $\bR^2$. 
   
When $\der=0$, the inequality follows from Lemma~\ref{Gderivatives}. Given an order of derivation $\der> 0$,  let $\mass_{\tau,\der}$ be as   in Lemma~\ref{traslazione} when $m\in (0,\finetau]$ and  $\mass_{\tau,\der}=\mass_{\tau,\finetau}$  when  $\der>\finetau$. 
For each $\der> 0$,  define a smooth function $v_\der$ on $ \bR^2$ by 
$$
v_\der(\xi_1,\xi_2) 
=u(\xi_1,\xi_2+\xi_1\mass_{\tau,\der})
\qquad \forall \xi_1,\xi_2 \in \bR^2,
$$
so that, 
when $t\geq 0$ and $s=0,1,\dots,a_\tau$,
$$ 
\cG_\tau f_\tau(t,\aut{s}t)
=u\big(t,\aut{s}t )\big)
=v_\der(t,(\aut{s}-\mass_{\tau,\der})t ).
$$ 
Let    
\[
v_\der\sim \sum_{p,q}\frac{a_{p,q}}{p!q!}\,t^p\xi_2^q\ .
\]
be   the Taylor expansion at the origin 
of  the smooth function  $v_\der$.
Our next goal is to write the Taylor coefficients $a_{p,q}$ of $v_\der$ in terms of derivatives of  $ \cG_\tau f_\tau$ at the origin.
For the sake of brevity, we  write~\eqref{e:lambdaitau} as
$$
\la_s
=\aut{s}-\mass_{\tau,\der}\qquad  s=0,1,\dots, a_\tau\ ,
$$
and we have  
\beas
c_{s,\tau,\der}(\cG_\tau f_\tau)
&=
\left(\frac{d}{dt}\right)^\der_{\big|_{t=0^+}}
v_\der(t,\la_st )
\\ &=
\left(\frac{d}{dt}\right)^\der_{\big|_{t=0^+}}
 \sum_{p+q=\der} \frac{a_{p,q}}{p!q!}t^\der\, \la_s^q 
\\ &=
 \sum_{p+q=\der}\binom{\der}{q} a_{p,q}  \, \la_s^q   .
\eeas
  \bigskip

   Therefore we are led to solve the  $(\finetau +1)\times (\der+1)$ linear system
\bea\label{sistema}
 \sum_{q=0}^{\der}  \, \la_s^q \, x_{q}=c_{s,\tau,\der}(\cG_\tau f_\tau)\qquad s=0,1,\dots, a_\tau
\eea
where 
\be\label{xh}
x_q=\binom{\der}{q} a_{\der-q,q} =\binom{\der}{q}\, \partial_1^{m-q}\partial_2^q v(0,0).
\ee
Let  $\der\leq a_\tau$. By Corollary~\ref{estensioneGn}   the system~\pref{sistema} is solvable, therefore we can  consider the first $\der+1$ equations.

Denote by 
$\Lambda$ the    $(\der+1)\times (\der+1)$    matrix  
   $\Lambda=\big(\Lambda_{s,q}\big)=\big(\la_s^q\big)$. Observe   that $\la_s\not=\la_j$ if $s\not=j$ and that
all  maximal minors of $\Lambda$ are essentially nonvanishing Vandermonde determinants.   
Thus, by Cramer's rule,
\bea\label{cramer}
|x_q|&\le\sum_{s=0}^{\der }  |c_{\tau, s,\der}(\cG_\tau f_\tau)|\,\left| \frac{V_{s,q}}V\right| \ ,\qquad q=0,1,\ldots, \der ,
\eea
where $V$ is the full Vandermonde determinant with nodes  $\la_s $
and $V_{s,q}$ are its cofactors. 
Expressing $V_{s,q}$ in terms of Schur polynomials,  cf. \cite{FH}, by Lemma~\ref{traslazione} we have
$$
\left| \frac{V_{s,q}}V\right| =\bigg|
\sum_{i_1<i_2<\cdots<i_{\der -q}\,,\,i_j\ne s}\frac{\la_{i_1}\la_{i_2}\cdots \la_{i_{\der -q}}}{\displaystyle\prod_{i\ne s}(\la_i-\la_s)}
\bigg|\le 
 \binom{\der}{q}\, 
\displaystyle\prod_{i\ne s, \la_i\ne 0}
\frac{\left|\la_i\right|} {
 \left|\la_i-\la_s \right|
} \leq  \binom{\der}{q}\, 
\ .
$$
Hence by \eqref{xh}, \eqref{cramer} and  Lemma~\ref{Gderivatives}   
$$
|\partial_1^{m-q}\partial_2^q v(0,0)|
 \leq \sum_{s=0}^{\der }  
 |c_{\tau, s,\der}(\cG_\tau f_\tau)|\leq   C_m \,  (\der+1)\,  \| f_\tau\|_{(N_\der)}.
$$
Note that 
$$
\partial_1^{\der-q}\partial_2^q u(0,0)
= 
\sum_{i=0}^{\der-q}   \binom{{\der-q}}{ i}  \mass_{\tau,\der}^i
\,\,
\partial_1^{{\der-q}-i}
\partial_2^{q+i} v_\der(0,0).
$$
Thus,  
\beas
\binom {\der}{ q}|\partial_1^{\der-q}\partial_2^q u(0,0)|  
&\leq 
d_\tau^\der\,  \binom{\der}{ q}\,
\sum_{i=0}^{\der-q}\binom{{\der-q}}{i}  
|\partial_1^{{\der-q}-i}
\partial_2^{q+i} v_\der(0,0)|
\\&
\leq
d_\tau^\der\,  \binom{\der}{ q}\,
\sum_{i=0}^{\der-q}\binom{{\der-q}}{ i}  
\, (\der+1)\,  
 |c_{\tau, s,\der}(\cG_\tau f_\tau)|.
 \\&
\leq
d_\tau^\der\, m^{m/2}\,
2^m
\, (\der+1)\,  C_m \,\| f_\tau\|_{(N_\der)}\ .
\eeas
 The case $\der> a_\tau$  can be reduced to the previous one by looking for a solution $x=(x_0,x_1,\ldots ,x_{q},\ldots x_{\der})$ with $x_{q}=0$ for $q>a_\tau$.

\bigskip

\subsubsection{Proof of Proposition~\ref{taylor} when $n=3$}\label{s:proof643}
\quad

When $n=3$ we can slightly modify the arguments used for the complex motion group on $\bC^2$  \cite{ADR4}. 
Indeed, in this case there is no need to use the technical Lemma~\ref{traslazione}, so the proof is less involved.
However, in order to have a solvable system of equations, one has to deal
 separately with the odd and the even part of  
  \hbox{$F\in \Smatr$}.

Given $f_\tau\in \Stau$, let \hbox{$F=A_\tau f_\tau$}, and  denote by $\Fo$ and $\Fe$ its odd and its even part, respectively.
From Corollary~\ref{Phi3} we have
$$ 
\cG_\tau^\sharp F_\pm(t^2, \aut{s}t)=\pm\,\cG_\tau^\sharp F_\pm(t^2,- \aut{s}t)  \qquad \forall (t^2,\aut{s}t)\in \Sigma_\cD^\tau.
$$
Therefore  we can restrict ourselves to consider $s=\mu, \mu+1,\dots 2\mu$ corresponding to $\aut{s}=-\mu+s=0,1,\dots, \mu$.

 If    $u$ is a smooth extension  of $\cG_\tau f_\tau=\cG_\tau^\sharp F$ to $\bR^2$, 
 straightforward computations show that the smooth  functions $u_-$ and $u_+$, defined by 
$$
u_\pm (\xi_1,\xi_2)=\frac12 (u(\xi_1,\xi_2)\pm u(\xi_1,-\xi_2)) \qquad \forall (\xi_1,\xi_2)\in \bR^2,
$$
 extend $\cG_\tau^\sharp\Fo$ and $\cG_\tau^\sharp\Fe$, respectively.
 
  Let us consider  the odd part first. 
  If 
$$
u_-\sim
\sum_{p,q}\frac{a_{p,q}}{p!q!}\,\,\xi_1^p\,\xi_2^q
$$
is the Taylor expansion at the origin of $u_-$ then $a_{p,2q'}=0$, for every $p,q'\in \bN$, and  the coefficients $a_{p,q}$ must  satisfy the equalities
\bea\label{vincoli}
c_{\mu+j,\tau, \der}(\cG_\tau^\sharp\Fo)
&=\left(\frac{d}{dt}\right)^{\der}_{\big|_{t=0^+}}\cG_\tau^\sharp \Fo\big(t^2, jt \big)
\\
&=\left(\frac{d}{dt}\right)^{\der}_{\big|_{t=0}} \left(\sum_{p,q}\frac{a_{p,q}}{p!q!}\,\,t^{2p+q}\,j^q\right)\qquad \qquad j=0,1,2,\ldots,\mu
\\
&=\sum_{2p+q=\der}\frac{(2p+q)!}{p!q!}\,a_{p,q}\, j^q.
\eea
Note that  
$$
\cG_\tau^\sharp \Fo\big(\xi_1,0\big)=u_-(\xi_1,0)=0\qquad \forall \xi_1\geq 0,
$$
so the equality in \eqref{vincoli}   for $j=0$ is trivial.
Moreover, 
$$
c_{\mu+j,\tau,2\der'}(\cG_\tau\Fo)=\left(\frac{d}{dt}\right)^{2\der'}_{\big|_{t=0^+}}\cG_\tau^\sharp \Fo\big(t^2, jt \big)=0
\qquad \forall t\geq 0, \quad \forall \der'\geq 0,
$$
so  the equalities~\eqref{vincoli}
 are trivial when $\der$ is even.  Hence we  write $\der=2m'+1$, $q=2q'+1$ and $p=m'-q'$ and we obtain an $\mu \times ( m'+1)$ linear system  of the form
$$
c_{j,2\der'+1}=\sum_{q'=0}^{\der'}(j^2)^{q'}\, x_{q'} \qquad j=1,2,\dots \mu
$$
 where 
 $$
 c_{j,2\der'+1}=j^{-1}\,c_{\mu+j,\tau, 2\der'+1}(\cG_\tau\Fo)
 $$
 and
 $$
 x_{q'} =\frac{(2m'+1)!}{ (m'-q')! (2q'+1)!}a_{m'-q',2q'+1}.
 $$ 
We now apply   similar arguments to the even part $\Fe\in \Smatr$. If
$$
u_+\sim
\sum_{p,q}\frac{a_{p,q}}{p!q!}\,\,\xi_1^p\,\xi_2^q
$$
is the Maclaurin expansion of $u_+$ then  
$
a_{p,2q'+1}=0
$
  for every $p, q'\in \bN$
and the coefficients $a_{p,q}$ must  satisfy the equalities
\be\label{vincoliPari}
c_{\mu+j,\tau, \der}(\cG_\tau\Fe)
=\sum_{2p+q=\der}\frac{(2p+q)!}{p!q!}\,a_{p,q}\, j^q \qquad \qquad j=0,1,2,\ldots,\mu\ .
\ee
Note that if $m$ is odd and $m=2p+q$, then $q$ is odd and $a_{p,q}=0$. Moreover 
$$
c_{\mu+j,\tau,2\der'+1}(\cG_\tau\Fo)=\left(\frac{d}{dt}\right)^{2\der'+1}_{\big|_{t=0^+}}\cG_\tau^\sharp \Fo\big(t^2, jt \big)=0
\qquad \forall t\geq 0, \quad \forall \der'\geq 0.
$$
Hence the previous equalities~\pref{vincoliPari}
 are trivial when $\der$ is odd. 
 Thus we consider $m=2m'$, with $m'\in \bN$ and we observe that  the equality in \eqref{vincoliPari} for $j=0$ reduces to 
$$
a_{m',0}=\frac{(m')!}{(2m')!} \,c_{\mu,\tau, \der}(\cG_\tau\Fe).
$$ 
Therefore we are led to solve the $\mu \times ( m'+1)$  linear system 
$$
c_{j,2m'}=\sum_{q'=0}^{m'}(j^2)^{q'}\, x_{q'} \qquad j=1,2,\dots \mu
$$
 where 
 $$
 c_{j,2m'}=c_{\mu+j,\tau,2\der'}(\cG_\tau\Fo)
 $$
 and
 $$
 x_{q'} =\frac{(2m')!}{ (m'-q')! (2q')!}a_{m'-q',2q'}.
 $$
It is now possible to follow the lines of the proof given for $n=4$ in order to get the desired estimate. 
Note that  the coefficients of the matrix $\Lambda$ are $\Lambda _{j,q}=(j^2)^q$ and an easy computation gives 
 $$
\prod_{\substack{j=1 \\ j\not= s}}^{m'+1} \frac{j^2}{|j^2-s^2|}=2\, \frac{\binom{m'+1}{ s}}{\binom{m'+1+s}{ s}}\leq 1
\qquad s=1,2,\ldots m'+1
$$
which allows us to  obtain the required estimates in the Vandermonde determinants.
Thus  the  proof of Proposition~\ref{taylor} is complete.

\subsection{
Conclusion of the Proof of Theorem~\ref{main}
}\label{s:conclusion}
\quad

\medskip
 
In this section we prove condition (S') of Theorem~\ref{rendiconti} adapting the arguments in  \cite{ADR4}.
We include here the proofs for reader's convenience and we also    
 correct some inaccuracies  in the proof of~\cite[Proposition~7.7]{ADR4}.

We use a Borel argument to produce a smooth function with given Taylor's coefficients at the origin and with controlled Schwartz norms up to some fixed order.
  Indeed, we observed in the proof of Proposition~\ref{taylor} that $c_{s,\tau, \der}(\cG_\tau f_\tau)$ are related to Maclaurin's coefficients of any extension $u$ of $\cG_\tau f_\tau$ and in the same proposition we gave estimates for these coefficients. The following 
  lemma is  analogous to Lemma 7.6 in  ~\cite{ADR4} and completes the   proof of item (i) in the Introduction. 

 \begin{lemma}\label{jet-extension}
Let  $f$ be in $\cS(G)^{{\rm Int}(K)}$ and $M$ in $\bN$. Then, for every $\tau\in \widehat K$  there exists $h_{\tau,M}$ in  $C_c^\infty(\bR^2)$ such that
$$
c_{s,\tau,\der}(\cG_\tau f_\tau)=c_{s,\tau,\der}(\left(h_{\tau,M}\right)_{|_{\Sigma_\cD^\tau}})
\qquad s=0,1,\ldots a_\tau \quad  \forall\der\geq 0
$$
and $\|h_{\tau,M}\|_{(M)}$ is rapidly decaying in $\tau$.
\end{lemma}

\begin{proof}

Let $f\in \cS(G)^{{\rm Int}(K)}$,  $M\in \bN$. Let 
$\ph\in C^\infty_c(\bR^2)$  supported in  $\{\xi\in \bR^2\,\,:\,\, |\xi|\le1\}$ and  equal to 1 for $|\xi|\le \frac12$.
For any integers $\der,q$ with $0\le q\le \der$ let
$$
\psi_{\der,q}(\xi_1,\xi_2)=\ph(\xi_1,\xi_2)\,\xi_1^{\der-q}\xi_2^q\ .
$$
Fix $\tau\in \widehat K$ and let  $u_\tau$ be a smooth extension of $\cG_\tau f_\tau$ to $\bR^2$. 
Following the standard Borel's Lemma procedure, define
\[
h_\der(\xi)= \frac1{\der!}\sum_{q=0}^\der \binom{\der}{q}  \partial_1^{\der-q}\partial_2^q u_\tau(0)\,\eps_\der^\der\,\psi_{\der,q}(\xi/\eps_\der)
\]
and
\begin{equation}\label{borel}
h_{\tau,M}(\xi)=\sum_{\der\in\bN}h_\der(\xi)\ ,
\end{equation}
where   the coefficients $\eps_\der\in(0,1]$ will be   chosen afterwards so that the series in~\eqref{borel} converges normally in every $C^N$-norm and $\|h_{\tau,M}\|_{(M)}$   decays rapidly in $\tau$.

By Proposition \ref{taylor}, for every $N\in\bN$, 
$$
\|h_\der\|_{C^N}\le \frac{d_\tau^\der}{\der!}\,\eps_\der^{\der-N}\,  C_\der\,\|f_\tau\|_{(N_\der)}\,\sum_{q=0}^\der\|\psi_{\der,q}\|_{C^N}\ ,
$$
where   the  sequences $\{C_{\der}\}$ and $\{N_{\der}\}$ are independent of $\tau$ and    increasing.

With $\al_\der=C_\der\sum_{q=0}^\der\|\psi_{\der,q}\|_{C^{\der-1}}$, we choose
$$
\eps_\der=\eps_{\der,\tau,M}=\begin{cases}1&\text{ if }\der\leq M\\
\frac1{d_\tau!}\frac1{(1+\al_\der\|f_\tau\|_{(N_\der)})}&\text{ if }\der> M\ . \end{cases}
$$

Then
$$
\sum_{\der> M}\|h_\der\|_{C^{\der-1}}
\le \frac1{d_\tau!}\sum_{\der\ge M}  \frac{d_\tau^m}{\der!}\le \frac{e^{d_\tau}}{d_\tau!} .
$$ 

This implies that the series $\sum_{\der\in\bN}\|h_\der \|_{C^k}$ converges for every $k$, so that $h_{\tau,M}$ is smooth.
Moreover, since the sequence $\{N_{\der}\}$ is  increasing,
then, for $N\leq  M$, 
$$
\sum_{\der\leq M}\|h_\der\|_{C^N}
\le \sum_{\der\leq M}\|h_\der\|_{C^{M}}
\le  \sum_{\der\leq M}\frac{d_\tau^\der}{\der!}\,  C_\der\,\|f_\tau\|_{(N_\der)}\,\sum_{q=0}^\der\|\psi_{\der,q}\|_{C^M}
\le s_M\, d_\tau^M\,\|f_\tau\|_{(N_M)} ,
$$
where 
$\displaystyle s_M=\sum_{d\le M}\frac{ C_\der}{\der!}\,  \sum_{q=0}^\der\|\psi_{\der,q}\|_{C^M}$.
 Hence, since $\|f_\tau\|_{(N_M)}$ is rapidly decreasing in $\tau$,
  the function $h_{\tau, M}$ defined in \eqref{borel} with $\eps_\der=\eps_{\der,\tau,M}$  satisfies our requirements.
\end{proof}

The following proposition is     analogous to Proposition 7.7 in  ~\cite{ADR4} and gives  (ii) in the introduction.
\begin{proposition}\label{p:nulla}
Let $f_\tau\in \Stau$ such that 
$$
c_{s,\tau, \der}(\cG_\tau f_\tau)=0\qquad \forall \der\in \bN \quad \forall s=0,1,\ldots a_\tau.
$$
Then $\cG_\tau f_\tau$ can be extended to a Schwartz function~$v_\tau$ satisfying the following norm bounds:
  for every $N\in \bN$ there exist constants $M_N\in \bN$ and $C_{N}>0$,  independent of $\tau$,\ 
such that 
$$
\|v_\tau\|_{(N)}\leq C_{N}\, \|f_\tau\|_{(M_N)}.
$$
\end{proposition}
\begin{proof} 
Let   $\eta$ be a bump function in $ C^\infty_c(\bR)$ supported in  $\big(-\frac12, \frac12\big)$ and equal to $1$ in a neighbourhood of the origin     Let $f_\tau\in \Stau$ and define   $v_\tau$ on $\bR^2$ by the rule
$$
v_\tau(\xi_1,\xi_2)=
\begin{cases}
\displaystyle \sum_{s=0}^{\finetau}  \cG_\tau f_\tau(\xi_1^{2/\undue},\aut{s} \xi_1)\, \eta_s \left(\tfrac{\xi_2}{\xi_1^{\undue/2}}\right)&\xi_1>0
\\
0&\xi_1\leq 0.
\end{cases}
$$
 where
 $$\eta_s(x)=\eta\left(x-\aut{s}\right), \qquad \forall x\in \bR\quad s=0,1,\cdots, \finetau \ .$$

It is straightforward to show that $v_\tau$ extends $ \cG_\tau f_\tau$ to $\bR^2$. We now show that $v_\tau$ is smooth and that for every $N\in \bN$ there exist constants $M_N\in \bN$ and $C_{N}>0$,  independent of $\tau$,
such that 
$$
\|v_\tau\|_{(N)}\leq C_{N}\, \|f_\tau\|_{(M_N)}.
$$

One can check by induction that, for appropriate coefficients $a_{r,h,q}$ 
$$
\partial^h_{\xi_1}\left(\xi_1^{-q\undue/2}\, \eta_s^{(q)}  \left(\tfrac{\xi_2}{\xi_1^{\undue/2}}\right)\right)
=\xi_1^{-q\undue/2-h}\,\sum_{r=0}^h a_{r,h,q}\,\left(\tfrac{\xi_2}{\xi_1^{\undue/2}}\right) ^r \eta_s^{(q+r)} \left(\tfrac{\xi_2}{\xi_1^{\undue/2}}\right).
$$
Now   fix $(\xi_1,\xi_2)$ with $\xi_1>0$ close enough to the spectrum so that there exists   $s$  such that 
$
\left|\frac{\xi_2}{\xi_1^{\undue/2}}- \aut{s}\right|<\frac12.
$
Then the sum defining $v_\tau$  reduces to a single term   and
\beas
\partial^p_{\xi_1}\partial^q_{\xi_2}v_\tau(\xi_1,\xi_2)
&=
 \sum_{h=0}^p\binom{p}{h}  
\partial_{\xi_1}^{p-h}\left(\cG_\tau f_\tau(\xi_1^{2/\undue},\aut{s}\xi_1)\right)\,\partial^h_{\xi_1}\left(\xi_1^{-q\undue/2}\, \eta_s^{(q)}  \left(\tfrac{\xi_2}{\xi_1^{\undue/2}}\right)\right).
 \\ &=
\sum_{h=0}^p  \sum_{r=0}^h \binom{p}{h}a_{r,h,q}\,\xi_1^{-q\undue/2-h}\,\left(\tfrac{\xi_2}{\xi_1^{\undue/2}}\right) ^r\,\,
\partial_{\xi_1}^{p-h}\left(\cG_\tau f_\tau(\xi_1^{2/\undue},\aut{s}\xi_1)\right)\,\, \eta_s^{(q+r)} \left(\tfrac{\xi_2}{\xi_1^{\undue/2}}\right) .
\eeas
Therefore we need to estimate 
\be\label{termine}
(1+|\xi_1|+|\xi_2|)^N\,\,
\xi_1^{-q\undue/2-h}\,\left|\tfrac{\xi_2}{\xi_1^{\undue/2}}\right| ^r\,\,
\partial_{\xi_1}^{p-h}\left(\cG_\tau f_\tau(\xi_1^{2/\undue},\aut{s}\xi_1)\right)\ ,
\ee
for every integer $N\geq 0$  when $
\left|\frac{\xi_2}{\xi_1^{\undue/2}}- \aut{s}\right|<\frac12.$

Note that
$$
|\xi_2|= |\xi_2-\xi_1^{\undue/2} \aut{s}+\xi_1^{\undue/2} \aut{s}|\leq|\xi_1^{\undue/2}|  \left( \tfrac12+|\aut{s}|\right)
$$
Thus the  expression in~\eqref{termine} can be controlled by a linear combination, independent of $\tau$, of terms 
of the type
 $$
\xi_1^{-\alpha} \,\partial_{\xi_1}^\ell \Big(
 {\xi_1}^{2\beta/\undue}\,( \xi_1\,\aut{s})^\gamma  \cG_\tau f_\tau(\xi_1^{2/\undue},\aut{s}\xi_1)\,\,
 \Big)=
 \xi_1^{-\alpha} \, \partial_{\xi_1}^\ell \Big(\cG_\tau (\Delta^\beta\,\mathbf D_\tau^\gamma f_\tau)(\xi_1^{2/\undue},  \aut{s}\xi_1)\Big)
$$
where $\alpha,\beta,\gamma,\ell$ are non negative integers.

Since $g(\xi_1)=\cG_\tau (\Delta^\beta\,\mathbf D_\tau^\gamma f)(\xi_1^{2/\undue},  \aut{s}\xi_1)$ vanishes of infinite 
order at the origin,
 for any integer $q\geq 0$ there exists $\theta \in (0,1)$
  such that
 for any $\xi_1> 0$
\[
\xi_1^{-\alpha}\, g^{(\ell)}(\xi_1)=\frac1{\ell!}\,
 g^{(\ell+\alpha)}(\theta \xi_1) .
\]
Therefore, in view of  Lemma~\ref{Gderivatives},  $v_\tau\in C^\infty(\bR^2)$ and the required norm estimates follow.
\end{proof}

Finally we collect all our results in order to produce an extension with norms rapidly decaying in $\tau$, so proving the Schwartz correspondence for $\big(M_n(\bR),SO_n\big)$ according to Theorem~\ref{rendiconti}. Given $f$ in $\cS(G)^{{\rm Int}(K)}$ and $\tau$ in $ \widehat K$, we have $\cG  f_\tau(\xi'_\tau,\xi'') =\cG_{\tau }  f_\tau (\xi'') $, hence condition (S') is equivalent to the following.

\begin{proposition}\label{S'}
Let  $f$ be in $\cS(G)^{{\rm Int}(K)}$ and $N$ in $\bN$. Then, for every $\tau\in \widehat K$, $\cG_{\tau }  f_\tau $
admits a Schwartz extension $u_{\tau,N}$ from $\Sigma_\cD^\tau$ to $\bR^2$ such that $\|u_{\tau,N}\|_{(N)}$ is rapidly decaying in~$\tau$.
\end{proposition} 
  
\begin{proof} Let $f$ be in  $\cS(G)^{{\rm Int}(K)}$, $\tau\in \widehat K$ and $N\in \bN$.
For $M$ to be chosen afterwards, let $h_{\tau,M}$ be the function associated to $f_\tau$, according to Lemma~\ref{jet-extension}. By Corollary~\ref{nucleo-tipo-tau} there exists  $g_{\tau,M}\in \Stau$ such that $\cG_\tau  g_{\tau,M}=
\left(h_{\tau,M}\right)_{|_{\Sigma_\cD^\tau}}.
$
Since
$$
c_{s,\tau,m}\left(\cG_\tau(f_\tau- g_{\tau,M})\right)=0\qquad \forall \der\in \bN \quad \forall s=0,1,\ldots a_\tau,
$$
by 
Proposition~\ref{p:nulla},here exists an extension $v_{\tau,M}$ of $\cG_\tau(f_\tau- g_{\tau,M})$ and a positive  integer  $N'$, depending only on $N$, 
 such that 
  $$
  \|v_{\tau,M}\|_{(N)}\leq \, C_N\,   \|f_\tau- g_{\tau,M}\|_{(N')}.
$$
By Corollary~\ref{nucleo-tipo-tau}, there exists $N''$, depending on $N'$, such that
$$
 \|g_{\tau,M}\|_{(N')}\leq C_{\tau,N}\|h_{\tau,M}\|_{(N'')}
$$ 
where the constant $C_{\tau,N}$ has polynomial growth in $\tau$.
Choose $M$ bigger than $N$  and $N''$, then
\[
u_{\tau,N}=v_{\tau,M}+h_{\tau,M}
\]
is a Schwartz extension to $\bR^2$ of $\cG_\tau f_\tau$ whose $N$-Schwartz  norm  decays rapidly in $\tau$. 
Indeed,
\beas
\|u_{\tau,N}\|_{(N)}& 
\leq C_N\,   \|f_\tau\|_{(N')}+C_N\,C_{\tau,N}\| h_{\tau,M} \|_{(N'')}+\| h_{\tau,M} \|_{(N)} 
\\ &
\leq C'_{\tau,N} \left( \|f_\tau\|_{(N')}+\| h_{\tau,M} \|_{(M)}\right)
\eeas
where the constant $C'_{\tau,N}$ has polynomial growth in $\tau$ and both  $\|f_\tau \|_{(N')}$ and $\| h_{\tau,M} \|_{(M)} $ decay rapidly in $\tau$.
\end{proof}

\end{document}